\documentclass{amsart}

\usepackage[all]{xy}

\usepackage{latexsym}

\usepackage{amssymb}
\usepackage{amsfonts}
\usepackage{amscd}
\usepackage{amsmath,amsthm}
\usepackage{xcolor}

\newtheorem{lemma}{Lemma}[section]
\newtheorem{proposition}[lemma]{Proposition}
\newtheorem{theorem}[lemma]{Theorem}
\newtheorem{corollary}[lemma]{Corollary}

{

}

\theoremstyle{definition}

\newtheorem{example}[lemma]{Example}
\newtheorem{definition}[lemma]{\sl Definition}

\theoremstyle{remark}

\newtheorem{remark}[lemma]{Remark}

\newcommand{\Hom}{\operatorname{Hom}}
\newcommand{\Ext}{\operatorname{Ext}}

\newcommand{\End}{\operatorname{End}}

\newcommand{\cO}{\mathcal{O}}
\newcommand{\rk}{\operatorname{rk}}
\newcommand{\cE}{\mathcal{E}}
\newcommand{\cF}{\mathcal{F}}

\numberwithin{equation}{section}

\begin{document}

\title[Three-periodic helices on elliptic curves]{Three-periodic helices on elliptic curves and their associated regular algebras}

\author{D. Chan}
\address{University of New South Wales}
\email{danielc@unsw.edu.au}

\author{A. Nyman}
\address{Western Washington University}
\email{adam.nyman@wwu.edu}

\begin{abstract}
Let $k$ denote an algebraically closed field of characteristic zero and let $X$ denote a smooth elliptic curve over $k$.  Given a three-periodic elliptic helix $\underline{\mathcal{E}}$ of vector bundles over $X$ with endomorphism $\mathbb{Z}$-algebra $\End \underline{\cE}$ and quadratic cover $\mathbb{S}^{nc}(\underline{\cE})$, we prove that $\End \underline{\cE}$ is the quotient of $\mathbb{S}^{nc}(\underline{\cE})$ by a degree three family of normal elements, generalizing our earlier work 
to the case in which $\operatorname{dim }\End \underline{\cE}_{i, i+1}$ isn't a constant function of $i$.  We then show that $\End \underline{\cE}$ is noetherian if and only if it has polynomial growth, and in this case, the ranks of any three consecutive bundles in the helix are a Markov triple.  Furthermore, in this case $\mathbb{S}^{nc}(\underline{\cE})$ is a noetherian GK-three $\mathbb{Z}$-algebra which is ${\sf Proj }$-equivalent to an elliptic algebra.  We conclude the paper by constructing several new families of elliptic helices with exponential growth.
\end{abstract}

\maketitle

\pagenumbering{arabic}
Throughout this paper, we let $k$ denote an algebraically closed field of characteristic zero and we let $X$ denote a smooth elliptic curve over $k$.

\section{Introduction}  \label{sec:intro}
In the classic paper \cite{atv}, Artin-Tate-Van den Bergh studied regular algebras associated to an  elliptic curve $X$ as noncommutative analogues of the polynomial ring in three variables and hence, also noncommutative planes. One way to interpret their construction is that they embed $X$ ``noncommutatively'' by replacing the degree three polarizing line bundle with an invertible bimodule $\cE$. This gives a noncommutative homogeneous coordinate ring, and taking the quadratic cover (that is, retaining only quadratic relations), one obtains a regular algebra. Such algebras have an associated division ring of fractions, which has transcendence degree two, and Artin conjectured in \cite{problems}, that the only other such division rings are those which are Ore extensions of transcendence degree one fields and those which are finite over their centre. 

Inspired by \cite{markov}, in which the notion of a helix of coherent sheaves over $\mathbb{P}^{2}$, i.e. a planar helix, was defined and studied in order to classify strongly exceptional collections of vector bundles over $\mathbb{P}^{2}$ of length three, Bondal and Polishchuk replaced the powers $\cE^{\otimes i}$ of $\cE$ by a sequence of line bundles $\underline{\cE} =(\cE_i)_{i\in \mathbb{Z}}$ satisfying a three-periodic helical condition \cite{bp} (see Definition~\ref{def:weakHelix}). A new feature that occurs is that the graded homogeneous coordinate ring is now replaced by the endomorphism ring $\End \underline{\cE} = B = \oplus B_{ij}, \ B_{ij} := \Hom_X(\cE_{-j}, \cE_{-i})$ which is no longer a graded algebra, but rather is a $\mathbb{Z}$-algebra.   

In \cite{CN} and \cite{morph}, we further generalised Bondal and Polishchuk's work to certain helices of vector bundles over $X$. One constraining feature in these papers was that the algebras constructed were all equigenerated in the sense that all the Hilbert series $H_{B,n} := \sum \dim_k B_{n,n+i} t^i$ were identical for all $n$. In this work, we study the non-equigenerated case in detail. 

Key to the analysis is a computation of the Hilbert series, of which there are now three (see Theorem~\ref{thm:hilbertSeries}):
\begin{theorem}  \label{thm:introHilbert}
Let $\underline{\cE}$ be an elliptic helix of vector bundles on $X$, $B = \End \underline{\cE}$ and $\mathbb{S}=\mathbb{S}^{nc}(\underline{\cE})$ its quadratic cover. Then the Hilbert series are 3-periodic in the sense that $H_{B,n+3} = H_{B,n}, H_{\mathbb{S},n+3} = H_{\mathbb{S},n}$. Furthermore, 
\begin{equation}
\begin{pmatrix}
    H_{\mathbb{S},0}(t) \\ H_{\mathbb{S},1}(t) \\ H_{\mathbb{S},2}(t) 
\end{pmatrix}
 = 
 \frac{1}{D(t)}
 \begin{pmatrix}
     1 \\ 1 \\ 1
 \end{pmatrix}
\ \text{ and } \ 
\begin{pmatrix}
    H_{B,0}(t) \\ H_{B,1}(t) \\ H_{B,2}(t) 
\end{pmatrix}
 = 
 \frac{1-t^3}{D(t)}
 \begin{pmatrix}
     1 \\ 1 \\ 1
 \end{pmatrix}.
\end{equation}
where $D(t)$ is the $3\times 3$-matrix given in (\ref{eq:denominatorHilbert}) of Theorem~\ref{thm:hilbertSeries}. 
\end{theorem}

It follows from \cite{bp} that $\mathbb{S}$ is Artin-Schelter regular of global dimension three (see Theorem~\ref{thm:BPBisKoszul}), 
and from our earlier work \cite{morph}, that $\End\underline{\cE}$ is the quotient of $\mathbb{S}$ by a degree 3 family of normal elements (see Corollary~\ref{cor:cubicDivisor}). This generalises the fact that elliptic curves embed as a cubic divisor in $\mathbb{P}^2$ as well as Artin-Tate-Van den Bergh's analogous noncommutative graded result. 

Our general setup raises the possibility of generating counter-examples to Artin's conjecture. Towards this end, we examine the noetherian case where one can form rings of fractions.  Before describing our main result, we recall that a {\it Markov triple} is a three-tuple $(m_{1},m_{2},m_{3})$ of positive integers such that 
$$
m_{1}^{2}+m_{2}^{2}+m_{3}^{2}=3m_{1}m_{2}m_{3}.
$$
We also recall (for later use) that there is a correspondence between Markov triples and positive integer solutions of 
\begin{equation} \label{eqn.markovtrip}
x^{2}+y^{2}+z^{2}=xyz
\end{equation}
defined by multiplication by $3$ \cite[Proposition 2.2]{hundred}.  We call a triple of positive integers satisfying (\ref{eqn.markovtrip}) a {\it scaled Markov triple}.

Our main result is a classification of three periodic elliptic helices $\underline{\mathcal{E}}$ such that $\End \underline{\cE}$ is noetherian.  We accomplish this in Proposition \ref{prop:noetherianIsMarkov}, Corollary \ref{cor.leftrightnoetherian} and Corollary \ref{cor:MarkovIsMutateLines} as follows:
\begin{theorem} \label{thm.elliptic}
Let $\underline{\mathcal{E}}$ denote a three-periodic elliptic helix. Then $\End \underline{\cE}$ is noetherian if and only if it has polynomial growth. Furthermore, in this case, $r_0:=\rk \mathcal{E}_{0}$, $r_1:=\rk \mathcal{E}_{1}, r_2:=\rk \mathcal{E}_{2}$ form a Markov triple, that is, $r_0^2 + r_1^2 + r_2^2 = 3 r_0r_1r_2$ and 
$$\dim \Hom(\cE_1,\cE_2) = 3r_0, \ \dim \Hom(\cE_2,\cE_3) = 3r_1, \ \dim \Hom(\cE_0,\cE_1) = 3r_2. 
$$
Lastly, $\mathbb{S}^{nc}(\underline{\mathcal{E}})$ is a noetherian GK-dimension three $\mathbb{Z}$-algebra, which is ${\sf Proj}$-equivalent to an elliptic algebra.
\end{theorem}
In particular, we do not find a counterexample to Artin's conjecture.  The appearance of Markov triples is, in retrospect, not surprising, since mutation of the foundations of a helix yields another helix (see \cite{markov}).   For essentially the same reason, Markov numbers appear in \cite{numvdb}.

We are also able to analyze the planar helices from \cite{markov} using our techniques.  In particular, we prove, in Section \ref{section.helicesP2}, that all planar helices are ample and restrict to elliptic helices.  However, the elliptic helices that arise in this manner correspond to the usual embedding of an elliptic curve in $\mathbb{P}^{2}$.

Our final task is to generate further examples of elliptic helices $\underline{\cE}$ of vector bundles with exponential growth. The helix is completely determined, through a recursive procedure, by the {\em seed} or {\em foundation} $(\cE_0,\cE_1,\cE_2)$ which is a triple of stable bundles of increasing slope. However, not all such triples give rise to elliptic helices and determining which ones do involves a priori checking that an infinite number of recursively defined objects in the derived category are actually bundles. In Proposition~\ref{prop:growthCriterionForGeneration}, we give a relatively effective sufficient criterion for a seed to generate an elliptic helix. It is in terms of the {\em numerical seed}
$([\cE_0],[\cE_1],[\cE_2])$ where $[\cE] := {\deg(\cE) \choose \rk(\cE)}$. It is strong enough to generate the following elliptic helices.

\begin{proposition}
Consider one of the following numerical seeds:
\begin{enumerate}
    \item $({ -r_0-r_2-a \choose r_0}, {0 \choose 1}, {r_0+r_2+a \choose r_2})$ for some $a,r_0,r_2 \in \mathbb{Z}_{\geq 1}$ such that $r_0,r_2 \geq a+1, \ r_0+r_2 > a^2+a+1$ and, of course $\gcd(r_0+r_2+a, r_i) = 1$ for $i = 0,2$, OR
    \item $({ -d-a-1 \choose d+a-r}, {0 \choose 1}, {d \choose r})$ where $a,d,r \in \mathbb{N}$ are chosen so that $\gcd(d+a+1,d+a-r) = 1 = \gcd(d,r)$. Suppose further that $d>(a+1)(2r+1)+4$ and $\min\{ d+a-r, r \} > 2$. 
\end{enumerate}
In either case, we obtain an elliptic helix with that numerical seed.  
\end{proposition}
\vskip .1in
\noindent{\it Notation and conventions:}  Throughout, ${\sf C}$ will denote a $k$-linear, $\operatorname{Hom}$-finite abelian category.  We will often write ${}^{j}(\mathcal{A}, \mathcal{B})$ for $\operatorname{Ext}^{j}_{\sf C}(\mathcal{A},\mathcal{B})$, and we will let ${\sf Coh }Y$ (resp. ${\sf Qcoh }Y$) denote the category of coherent (resp. quasicoherent) sheaves over a variety $Y$.  If ${\sf D}$ and ${\sf E}$ are categories, we will write ${\sf D} \equiv {\sf E}$ if there is an equivalence of categories between ${\sf D}$ and ${\sf E}$.

\section{Basic notions from noncommutative algebraic geometry}
We will show that planar helices and elliptic helices ``of Markov type" are ample in Lemma \ref{lem:twistBundleAmple} and Theorem \ref{prop.ellipticample}, respectively.  We will use this fact to prove that the algebras associated to such helices $\underline{\mathcal{E}}$, $\mathbb{S}^{nc}(\underline{\mathcal{E}})$ and $\operatorname{End }(\underline{\cE})$, are noetherian.  For the reader's convenience, we recall the relevant definitions and results from the theory of $\mathbb{Z}$-algebras and noncommutative algebraic geometry (in the $\mathbb{Z}$-algebra context) here.  The reader wishing for more detail should consult \cite{morinyman} and \cite{MN}.

To begin, we recall that a {\it $\mathbb{Z}$-algebra} (over a field $k$) is a ring with vector space decomposition $C=\oplus _{i, j\in \mathbb{Z}}C_{ij}$ and with the multiplication
$$C_{ij}\otimes C_{st}\to \begin{cases} C_{it} & \textnormal { if } j=s \\ 0 & \textnormal { if } j\neq s, \end{cases}$$ such that $C_{ii}$ is a $k$-algebra for every $i$ and such that $k$ acts centrally on $C_{ij}$ for all $i,j$.  If $C$ is a $\mathbb{Z}$-algebra and $q$ is a nonnegative integer, we let $C_{\geq q}$ be the $\mathbb{Z}$-algebra $\bigoplus_{q \leq j-i}C_{ij}$.

We say a $\mathbb{Z}$-algebra $C$ is 
\begin{itemize}
\item{} {\it positively graded} if $C_{ij}=0$ whenever $i>j$, 
\item{} {\it connected} if $C$ is positively graded and $C_{ii}=k$ for all $i \in \mathbb{Z}$, and
\item{} {\it locally finite} if $\operatorname{dim }_{k} C_{ij} < \infty$ for all $i,j$.
\end{itemize}

A {\it graded right $C$-module} is a graded vector space $M=\oplus _{j\in \mathbb{Z}}M_j$ with a $C$-module  action $M_i\otimes C_{ij}\to M_j$ satisfying the usual module axioms. In particular, we assume that each $M_i$ is a unitary $C_{ii}$-module in the sense that $me_i=m$ for every $m\in M_i$.  The category of graded right $C$-modules is denoted by ${\sf Gr} C$ whose morphisms are right $C$-module homomorphisms preserving degrees. A graded right $C$-module $M$ is {\it right bounded} if $M_{n} = 0$ for all $n >> 0$

Suppose $C$ is a connected, locally finite $\mathbb{Z}$-algebra.  We let ${\sf Tors }C$ denote the full subcategory of ${\sf Gr }C$ consisting of modules whose elements $m$ have the property that the right $C$-module generated by $m$ is right bounded.  If ${\sf Tors }C$ is a localizing subcategory of ${\sf Gr }C$, which occurs, for example, if $C$ is generated in degree one, then we may form the quotient ${\sf Gr }C/{\sf Tors }C =: {\sf Proj }C$.  We let $\pi:{\sf Gr }C \longrightarrow {\sf Proj }C$ denote the quotient functor.

\begin{definition} Let $C$ be a $\mathbb{Z}$-algebra.
\begin{enumerate}
\item{} We say that $M\in {\sf Gr} C$ is {\it finitely generated} (resp. {\it finitely presented}) if there exists an exact sequence $F^0\to M\to 0$ (resp. $F^1\to F^0\to M\to 0$) where $F^i$ is a finite direct sum of objects of the form $e_{j}C$ for various $j$.

\item{} We say that $M\in {\sf Gr} C$ is {\it coherent} if $M$ is finitely generated and $\operatorname{Ker }\phi$ is finitely generated for every homomorphism $\phi: F\to M$ with $F$ a finite direct sum of objects of the form $e_{j}C$ for various $j$.
\item{} We denote by ${\sf coh} C$ the full subcategory of ${\sf Gr} C$ consisting of coherent modules.  
\item{} We say that $C$ is {\it right coherent} if $e_iC$ and $e_iC/(e_{i}C)_{\geq 1}$ are objects in ${\sf coh} C$ for every  $i\in \mathbb{Z}$.
\item{} We say that $C$ is {\it right noetherian} if $e_{j}C$ is a noetherian object for every $j\in \mathbb{Z}$.
\end{enumerate}
\end{definition}
By \cite[Proposition 1.1]{polishproj}, if $C$ is connected and locally finite, ${\sf coh }C$ is an abelian subcategory of ${\sf Gr }C$ closed under extensions.  In this case, we let ${\sf tors }C$ denote the full subcategory of ${\sf coh }C$ consisting of right-bounded modules.  One can check that this is a Serre subcategory of ${\sf coh }C$.  If $C$ is right coherent, we define
$$
{\sf proj }C := {\sf coh }C/{\sf tors }C,
$$
which is abelian.  We abuse notation by letting $\pi:{\sf coh }C \longrightarrow {\sf proj }C$ denote the corresponding quotient functor.

The definition of ampleness we will need is (superficially) distinct from that in \cite{polishproj}, as our indexing conventions are different: we define the $ij$-component of the $\mathbb{Z}$-algebra of a helix $\underline{\mathcal{E}}$, which we call $\operatorname{End}(\underline{\mathcal{E}})$, as $\operatorname{Hom}(\mathcal{E}_{-j}, \mathcal{E}_{-i})$, whereas in \cite{polishproj}, this component is defined to be $\operatorname{Hom}(\mathcal{E}_{i}, \mathcal{E}_{j})$.  For the remainder of the section, we let $\underline{\mathcal{E}} = (\mathcal{E}_{i})_{i \in \mathbb{Z}}$ denote a sequence of objects in ${\sf C}$ such that $\operatorname{Hom}(\mathcal{E}_{i},\mathcal{E}_{i})=k$.

We call $\underline{\mathcal{E}}$
\begin{itemize}
\item{} {\it projective} if for every epimorphism $f: \mathcal{M} \rightarrow \mathcal{N}$ in ${\sf C}$ there exists an integer $n$ such that $\operatorname{Hom}_{\sf C}(\mathcal{E}_{-i},f)$ is surjective for all $i>n$, and

\item{} {\it ample for ${\sf C}$} if it is projective, and if for every $\mathcal{M} \in {\sf C}$ and every $m \in \mathbb{Z}$ there exists an epimorphism
    $$
    \bigoplus_{j=1}^{s}\mathcal{E}_{-i_{j}} \rightarrow \mathcal{M}
    $$
for some $i_{1},\ldots, i_{s}$ with $i_{j} \geq m$ for all $j$.
\end{itemize}
The significance of this concept for us is that ample sequences for ${\sf C}$ yield homogeneous coordinate rings of ${\sf C}$ \cite[Theorem 2.4]{polishproj}.  

Before we prove our version of \cite[Proposition 2.3(ii) and Theorem 2.4]{polishproj} and \cite[Theorem~4.5]{az} in Theorem \ref{theorem.polish}, we recall that if a collection $\underline{\mathcal{E}}$ is ample, $i \in \mathbb{Z}$ and 
$$
\Gamma_{\geq i}:{\sf C} \longrightarrow {\sf coh }\operatorname{End}(\underline{\mathcal{E}})_{\geq 0}
$$ 
is defined by $\Gamma_{\geq i}(-) = \oplus_{n \leq -i} \Hom_{\sf C}({\mathcal{E}}_n,-)$, then $\pi \Gamma_{\geq i}: {\sf C} \longrightarrow {\sf proj }\operatorname{End}(\underline{\mathcal{E}})_{\geq 0}$ is an equivalence by \cite[Proposition 2.3(ii) and Theorem 2.4]{polishproj}.  We will also need the following result.

\begin{lemma} \label{lemma.tor}
If $\underline{\mathcal{E}}$ is an ample sequence in {\sf C} then $\Gamma_{\geq i} X$ is torsion-free for every $X \in {\sf C}$. 
\end{lemma}
\begin{proof}
Consider $0 \neq \phi \in \Hom_{\sf C}(\mathcal{E}_{j},X) \subseteq \Gamma_{\geq i} X$ for some $j \leq -i$. By ampleness, given any $m < j$ we can find $j_1,\ldots, j_l <m$ and an epimorphism $\Psi\colon \oplus_r \mathcal{E}_{j_r} \to \mathcal{E}_j$ so $\phi\circ \Psi \neq 0$. It follows that there is a component $\psi\colon \mathcal{E}_{j_r} \to \mathcal{E}_j$ of $\Psi$ such that $\phi\circ \psi \neq 0$ so $\Gamma_{\geq i} X$ is indeed torsion-free. 
\end{proof}

\begin{theorem} \label{theorem.polish}
If $\underline{\mathcal{E}}$ is an ample sequence then $\operatorname{End}(\underline{\mathcal{E}})_{\geq 0}$ is right coherent, and there is an equivalence
$$
{\sf C} \equiv {\sf proj }\operatorname{End}(\underline{\mathcal{E}})_{\geq 0}.
$$
If furthermore, the category {\sf C} is noetherian, then $\operatorname{End}(\underline{\mathcal{E}})_{\geq 0}$ is right noetherian. 
\end{theorem}
\begin{proof}
The coherence of $A:= \operatorname{End}(\underline{\mathcal{E}})_{\geq 0}$ follows from \cite[Proposition 2.3(ii)]{polishproj}, whilst the category equivalence is \cite[Theorem 2.4]{polishproj}. We now assume that {\sf C} is noetherian and show $A$ is right noetherian. To this end, let $\{I_j\}_{j \in \mathbb{N}}$ be an increasing chain of finitely generated submodules of $e_{i}A$. By the category equivalence and the fact that {\sf C} is noetherian, the chain $\{\pi(I_{j})\}$ stabilises in {\sf proj}~$A$ and we may assume without loss of generality that it stabilises at $j=1$.  Therefore, we may assume that all of the induced monomorphisms $\pi(I_{j}) \rightarrow \pi(I_{j+1})$ are isomorphisms in ${\sf proj }A$ for all $j \geq 1$.  Thus, the inclusions $I_{j} \rightarrow e_{i}A$ induce morphisms $\pi(I_{j}) \rightarrow \pi(e_{i}A)$ which have the same image for all $j$.  Under the equivalence $\pi \Gamma_{\geq i}$, this image is a subobject $J$ of $\mathcal{E}_{-i}$.

To the exact sequence 
$$0 \to J \to \mathcal{E}_{-i} \to \mathcal{E}_{-i}/J \to 0,$$
we apply the left exact functor $\Gamma_{\geq i}$ to obtain the exact sequence of $A$-modules
$$0 \to \Gamma_{\geq i}J \to e_{i}A \to \Gamma_{\geq i}(\mathcal{E}_{-i}/J).$$
For ease of notation, we let $K=\Gamma_{\geq i}J$.

We now claim that $K$ contains every $I_{j}$.  To prove the claim, we first note that, by Lemma \ref{lemma.tor}, $\Gamma_{\geq i}(\mathcal{E}_{-i}/J)$ is torsion-free, and it follows that $e_{i}A/K$ is torsion-free as well.  We will show that $\frac{I_{j}+K}{K}$ is a torsion submodule of $e_{i}A/K$, so that $\frac{I_{j}+K}{K}=0$, whence $I_{j} \subset K$.  Thus, to prove the claim, we establish that $\frac{I_{j}+K}{K}$ is torsion.  To this end, let $i_{j}:I_{j} \rightarrow e_{i}A$ and $\kappa: K \rightarrow e_{i}A$ denote inclusions, and let $p:e_{i}A \rightarrow e_{i}A/K$ denote projection.  Since $\pi(i_{j})$ and $\pi(\kappa)$ represent the same subobject of $\pi(e_{i}A)$, there exists an isomorphism $\alpha_{j}:\pi(I_{j}) \rightarrow \pi(K)$ such that $\pi(i_{j})=\pi(\kappa) \alpha_{j}$.  Thus, $\pi(p i_{j})=\pi(p) \pi(\kappa)\alpha_{j}=0$.  By exactness of $\pi$, 
$$
\mbox{im} (p i_{j}) = \frac{I_{j}+K}{K}
$$
is thus torsion.   Therefore, we have proven the claim that $K$ contains every $I_{j}$.  Finally, we note that $\pi(K/I_{j})=0$ for all $j$ by exactness of $\pi$ so that $K/I_{j}$ is finite-dimensional, which implies the original chain stabilizes.


\end{proof}

We next examine left finiteness conditions on $\operatorname{End}(\underline{\mathcal{E}})_{\geq 0}$.  To this end, we recall (from \cite{MN}) that if $C$ is a $\mathbb{Z}$-algebra, then the {\it opposite $\mathbb{Z}$-algebra of $C$} is the algebra denoted $C^o$ and with $C^o_{ij}:=C_{-j, -i}$, and with multiplication defined as follows:  for $a^o\in C^o_{jk}=C_{-k, -j}, b^o\in C^o_{ij}=C_{-j, -i}$, $b^oa^o=ab \in C_{-k, -i}=C^o_{ik}$.  We also note that, by \cite[Proposition 2.2]{MN}, the category of graded right $C^{o}$-modules is equivalent to the category of graded left $C$-modules.  In the next result, we will invoke the term {\it duality} on ${\sf C}$, by which we mean a $k$-linear contravariant functor $D:{\sf C} \rightarrow {\sf C}$ such that $D^2 \cong \mbox{id}_{\sf C}$.

\begin{proposition} \label{prop.leftright}
Suppose $D$ is a duality on ${\sf C}$, and let $\underline{\mathcal{E}}$ be an ample sequence on ${\sf C}$.  Let $D(\underline{\mathcal{E}})$ denote the sequence of objects with $D(\underline{\mathcal{E}})_{i}:= D(\mathcal{E}_{-i})$.  If $D(\underline{\mathcal{E}})$ is ample, then $\operatorname{End}(\underline{\mathcal{E}})_{\geq 0}$ is left and right coherent.  Furthermore, if ${\sf C}$ is noetherian then $\operatorname{End}(\underline{\mathcal{E}})_{\geq 0}$ is left and right noetherian.     
\end{proposition}

\begin{proof}
We claim that if $E := \operatorname{End}(\underline{\mathcal{E}})_{\geq 0}$, then $E^{o} \cong \operatorname{End}(D(\underline{\mathcal{E}}))_{\geq 0}$.  From the claim, it will follow that since $D(\underline{\mathcal{E}})$ is ample, $E^{o}$ is right coherent by Theorem \ref{theorem.polish}.  Thus, $E$ is left coherent.  Similarly, if ${\sf C}$ is noetherian and $D(\underline{\mathcal{E}})$ is ample, then by the second part of Theorem \ref{theorem.polish}, $E^{o}$ is right noetherian, so that $E$ is left noetherian as desired.

It remains to prove the claim.  To do so, we note that for each $i,j$ there is an isomorphism
$$
E_{ij}^{o} = E_{-j,-i} = {}^{0}(\cE_{i}, \cE_{j}) \overset{D}{\longrightarrow} {}^{0}(D(\cE_{j}), D(\cE_{i})) = {}^{0}(D(\underline{\mathcal{E}})_{-j}, D(\underline{\mathcal{E}})_{-i})=\operatorname{End}(D(\underline{\mathcal{E}}))_{ij}.
$$
In addition, the above assignment is compatible with multiplication by functoriality and contravariance of $D$.
\end{proof}

We will also need the following $\mathbb{Z}$-algebra version of \cite[Lemma~8.2]{atv}.  The statement uses the notion of normal family of regular elements, which is defined in \cite[Definition 8.6]{morph}.  For the reader's convenience, we recall the definition now.

\begin{definition}
A {\it normal family of elements of  degree $d$} in a $\mathbb{Z}$-algebra $A$ is a family $\underline{g} = (g_i)_{i \in \mathbb{Z}}$ where $g_i \in A_{i,i+d}$ and such that for every $i,j$ we have $g_i A_{i+d,j+d} = A_{i,j} g_j$. An element $g \in A_{ij}$ is said to be {\it regular} if  right and left multiplication by $g$ are injective as maps from $Ae_i \to Ae_j$ and $e_jA \to e_iA$ respectively. 
\end{definition}

\begin{proposition} \label{prop.hbt}
Let $C$ be a positively graded, connected, locally finite $\mathbb{Z}$-algebra and $x = \{x_i \in e_i C\}_{i \in \mathbb{Z}}$ be a normal family of positive degree regular elements. If $C/(x)$ is right noetherian, then so is $C$. 
\end{proposition}
\begin{proof}
We present here the classic proof of Hilbert's basis theorem adapted to this setting. To simplify notation, we will use $x^n$ to denote an appropriate legal product of $n$ of the $x_i$'s. The choice of which $n$-fold product should be clear from the context. For example, if $c \in C_{ij}$, then the left-most $x$ in $cx^n$ lies in $e_j C$. 

We show that any submodule $I < e_i C$ is finitely generated. Consider the ``$x$-saturation'' 
$$\bar{I} = \{c \in e_i C| cx^N \in I, N \gg 0\}.$$
Normality of $x$ ensures that $\bar{I}$ is a submodule of $e_iC$. Now $e_iC/e_ixC$ is noetherian so we can find generators $a_1, \ldots, a_n$ for $\bar{I}$ modulo $e_i x C$. Pick $l_1,\ldots, l_n \in \mathbb{N}$ so that $a_1x^{l_1}, \ldots, a_nx^{l_n}\in I$ and let $l = \max \{l_i\}$. Let $b_1, \ldots, b_m\in I$ be generators for $I$ modulo $e_ix^l C$. 

We complete the proof by showing that $a'_1:=a_1x^{l_1}, \ldots, a'_n:=a_nx^{l_n},b_1, \ldots, b_m$ generate $I$.  Let $c \in I \cap e_i x^n C$. It suffices by induction to find $c' \in I \cap e_ix^{n+1}C$ such that $c-c'$ is a $C$-
linear combination of the $a'_i,b_j$ since the bidegree of $c$ and $c'$ are fixed. If $n <l$, then it suffices to use a linear combination of the $b_i$'s. If $n\geq l$, we write $c = c'' x^n$ where $c'' \in \bar{I}$ so we can write $c'' = \sum a_ic_i+ c'''x$. Using normality of $x$, we are done.
\end{proof}

\section{Helices} \label{section.elliptic}
The main focus of this paper is the study of helices.  In this section, we collect the relevant definitions of the term {\it helix} we will need.  In particular, we define weak helices, and recall the notions of elliptic helices and planar helices, from \cite[Section~7]{bp} and \cite{markov}, respectively.  In the next section, we relate the latter two notions.  

\subsection{Weak helices}  We first define the notion of {\it weak helix}, which is a generalization of the other notions of helix we will study.  To this end, let $V$ be a finite dimensional $k$-space and $\mathcal{M} \in {\sf C}$. 
\begin{definition} \label{def:abstractTensorWithVectorSpace}
Since {\sf C} has finite direct sums, there exists an object $V \otimes_k \mathcal{M} \in {\sf C}$ such that 
$$\Hom_{\sf C}(V \otimes_k \mathcal{M}, -) = V^* \otimes_k \Hom_{\sf C}(\mathcal{M},-).$$
It is unique up to canonical isomorphism.
\end{definition}
This definition has some advantages over a similar definition given in \cite{az2} such as the following. Suppose $\mathcal{L} \in {\sf C}$ and let $V = \Hom_{\sf C}(\mathcal{M},\mathcal{L})$. There is a canonical element
$$1 \in \End_k V \stackrel{\text{can}}{\simeq} V^* \otimes_k \Hom_{\sf C}(\mathcal{M},\mathcal{L})$$
which by Definition~\ref{def:abstractTensorWithVectorSpace} gives the {\em evaluation map}
$$
\operatorname{Hom}(\mathcal{M}, \mathcal{L}) \otimes_{k} \mathcal{M} \rightarrow \mathcal{L}.
$$
The notion of mutation (defined below) was first employed systematically to develop the theory of helices in \cite{markov}.
\begin{definition}  \label{def:mutable}
An ordered pair $(\cE,\cF)$ of objects in ${\sf C}$ is said to be {\it left mutable} if $^j(\cE,\cF) = 0$ for $j \neq 0$ and furthermore the evaluation map 
$$
\epsilon \colon \Hom(\cE,\cF) \otimes \cE \to \cF
$$
is an epimorphism. In this case we define the {\it left mutation} $L_\cE \cF := \ker \epsilon$. We also say {\it $\cF$ left mutates through $\cE$ in ${\sf C}$}. 

The right-handed versions are defined similarly.  Explicitly, the pair $(\cE,\cF)$ is {\it right mutable} if $^j(\cE,\cF) = 0$ for $j \neq 0$ and furthermore the coevaluation map
$$
\eta \colon \cE \to \Hom(\cE,\cF)^{*} \otimes \cF
$$
is a monomorphism.  In this case we define the {\it right mutation} $R_\cF \cE := \operatorname{cok }\eta$.
\end{definition}
For the rest of this section, we assume we have a sequence $\underline{\mathcal{E}} = (\mathcal{E}_{i})_{i \in \mathbb{Z}}$ of objects in ${\sf C}$ such that $\Delta_{j} := \operatorname{dim} \operatorname{Hom}(\mathcal{E}_{j-1}, \mathcal{E}_{j}) \neq 0$.  Suppose, for all $i \in \mathbb{Z}$:
\begin{enumerate} 
\item{} the pair $(\mathcal{E}_{i}, \mathcal{E}_{i+1})$ is right mutable, 

\item{} the pair $(R_{\mathcal{E}_{i+1}} \mathcal{E}_{i}, \mathcal{E}_{i+2})$ is right mutable and $R_{\mathcal{E}_{i+2}}R_{\mathcal{E}_{i+1}} \mathcal{E}_{i} \cong \mathcal{E}_{i+3}$, 

\item{} $\Hom(R_{\mathcal{E}_{i+1}}\mathcal{E}_{i}, \mathcal{E}_{i+2})^{*} \cong \Hom(\mathcal{E}_{i+2}, \mathcal{E}_{i+3})$.

\end{enumerate}
\begin{definition}  \label{def:weakHelix}
Sequences $\underline{\mathcal{E}}$ satisfying (1) and (2) above will be called {\em weak helices of period three}, {\em three-periodic weak helices}, or just {\em weak helices} when no confusion arises. We will also say that $\underline{\mathcal{E}}$ is {\em three-helical} in this case. 
\end{definition}
By property (1), there is an exact sequence
$$
0 \longrightarrow \mathcal{E}_{i} \longrightarrow \Hom(\mathcal{E}_{i}, \mathcal{E}_{i+1})^{*} \otimes \mathcal{E}_{i+1} \longrightarrow R_{\mathcal{E}_{i+1}}\mathcal{E}_{i} \longrightarrow 0,
$$
and by (2), there is an exact sequence 
$$
0 \longrightarrow R_{\mathcal{E}_{i+1}}\mathcal{E}_{i} \longrightarrow \Hom(R_{\mathcal{E}_{i+1}}\mathcal{E}_{i}, \mathcal{E}_{i+2})^{*}\otimes \mathcal{E}_{i+2} \longrightarrow \mathcal{E}_{i+3} \longrightarrow 0.
$$
Splicing these sequences together and utilizing (3) yields an exact sequence
\begin{equation} \label{eq:helixExactSeq}
0 \longrightarrow \mathcal{E}_{i} \longrightarrow \Hom(\mathcal{E}_{i}, \mathcal{E}_{i+1})^{*} \otimes \mathcal{E}_{i+1} \longrightarrow \Hom(\mathcal{E}_{i+2}, \mathcal{E}_{i+3}) \otimes \mathcal{E}_{i+2} \longrightarrow \mathcal{E}_{i+3} \longrightarrow 0.    
\end{equation}
Note that conditions (1)-(3) imply that $R_{\mathcal{E}_{i+1}}\mathcal{E}_{i} \simeq L_{\mathcal{E}_{i+2}}\mathcal{E}_{i+3}$ so there is a symmetry between left and right mutations. 


\subsection{Elliptic helices} \label{sub.elliptic}
We now define elliptic helices.  We will study this concept in more depth in Sections \ref{sec.numbericalSetup}, \ref{sec:hilbertSeries}, \ref{sec:seedsMarkovCase} and \ref{sec:ellipticHelixAmple}.

\begin{definition}  
An object $\cE \in {\sf C}$ is {\it elliptically exceptional} if i) $^j(\cE,\cE) \cong k$ for $j = 0,1$ and is zero otherwise and ii) for any $\cF \in {\sf C}$, the natural pairing 
$$
{}^{0}(\cE,\cF) \otimes \,^1(\cF,\cE) \to \,^1(\cE,\cE) = k
$$ 
is non-degenerate. 
\end{definition}
It follows (see \cite[Section 6]{morph}) that a coherent sheaf $\cE$ over $X$ is elliptically exceptional in ${\sf Coh }X$ if and only if  $\cE$ is either a vector bundle with relatively prime degree and rank or $\cE$ is a skyscraper sheaf.  

The following abelian category version of Bondal-Polishchuk's three-periodic elliptic helices is the same one used in \cite{morph}.
\begin{definition}  \label{def:ellipticHelix}
(\cite[Section~7, p. 250]{bp}) Let $\underline{\cE} = (\cE_i)_{i \in \mathbb{Z}}$ be a sequence of elliptically exceptional objects in {\sf C}. We say $\underline{\cE}$ is a {\it three-periodic elliptic helix} or just an {\it elliptic helix} if $\underline{\cE}$ is a three-periodic weak helix such that for all $i<j$ we have $^l(\cE_i,\cE_j) = 0$ for $l \neq 0$. 
\end{definition}

\begin{remark} \label{rem:ellipticHelix}
\begin{enumerate}
    \item By the remark following \cite[Question 7.1]{morph}, if $\underline{\cE}$ is an elliptic helix, then $\mathcal{E}_{i}$ is simple, i.e. $\operatorname{End}(\mathcal{E}_{i})=k$.  Hence $\mathcal{E}_{i}$ are stable bundles for all $i \in \mathbb{Z}$. Furthermore, the slopes of these bundles are increasing.  
    \item The analogous definition in terms of left mutations recovers the same concept.
    \item Property (3) preceding the definition of weak helix automatically holds for an elliptic helix \cite[Lemma 3]{rudakov}.
    \item Any weak helix of stable bundles on an elliptic curve $X$ is automatically an elliptic helix since the stable bundles must have strictly increasing slope so $^1(\cE_i, \cE_j) = 0$ for $i<j$. 
\end{enumerate}
\end{remark}

The significance of Definition \ref{def:ellipticHelix} is the following result from \cite{bp}, adapted to our context (\cite[Theorem 5.4]{morph}).

\begin{theorem}[\cite{bp}]  \label{thm:BPBisKoszul}
Let $\underline{\cE}$ be a three-periodic elliptic helix. Let $\mathbb{S}^{nc}(\underline{\cE})$ be the quadratic part of the endomorphism algebra $\End(\underline{\cE})$. Then the following hold:
\begin{enumerate}
    \item{} The algebra $\mathbb{S}^{nc}(\underline{\cE})$ is $3$-periodic, Koszul, has global dimension three, and is AS-regular of dimension three and Gorenstein parameter three (in the sense of \cite[Definition 7.1]{morinyman}).
    \item{} The canonical map $\mathbb{S}^{nc}(\underline{\cE}) \longrightarrow \End(\underline{\cE})$ is surjective so in particular, $\End(\underline{\cE})$ is generated in degree one. 
\end{enumerate}
\end{theorem}

\subsection{Planar helices}
We end this section by giving the definition of planar helix from \cite[Definition~p.80]{bondal}, which we will study in detail in Section \ref{section.helicesP2}.  First, we need to recall that a coherent sheaf $\mathcal{E}$ on a variety over $k$ is called {\it exceptional} if $\operatorname{dim }{}^{0}(\mathcal{E}, \mathcal{E})=1$ and $\operatorname{dim }{}^{i}(\mathcal{E}, \mathcal{E})=0$ for $i>0$. 

\begin{definition} A sequence $\underline{\mathcal{E}}$ of exceptional bundles on $\mathbb{P}^2$ is a {\it planar helix} if it is a three-periodic weak helix which further satisfies i) the periodicity relation $\mathcal{E}_{i-3} \simeq \mathcal{E}_i \otimes \omega_{\mathbb{P}^2}$, where $\omega_{\mathbb{P}^2}=\mathcal{O}_{\mathbb{P}^{2}}(-3)$ is the canonical bundle, and ii) any three consecutive terms $\mathcal{E}_i, \mathcal{E}_{i+1}, \mathcal{E}_{i+2}$ is exceptional in the sense that $\mathbf{R}\Hom(\mathcal{E}_j, \mathcal{E}_l) = 0$ whenever $i\leq l < j < i+2$. 
\end{definition}

\section{Planar helices and their associated $\mathbb{Z}$-algebras} \label{section.helicesP2}
In this section, we show that planar helices are ample, have noetherian endomorphism rings, and restrict to elliptic helices as defined in Section~\ref{section.elliptic}. 

\begin{lemma}  \label{lem:amplenessForBundles}
Let $\underline{\mathcal{V}} = (\mathcal{V}_i)$ be a sequence of non-zero vector bundles, each of fixed positive rank, on a projective scheme $Y$ such that for any $\mathcal{M} \in {\sf Coh}(Y)$, we have $\Ext^1_Y(\mathcal{V}_i, \mathcal{M}) = 0$ for $i \ll 0$. Then $\underline{\mathcal{V}}$ is ample in ${\sf Coh}(Y)$. 
\end{lemma}
\begin{proof}
Projectivity follows from our hypotheses and the long exact sequence in Ext. We show now there exists an epimorphism of the form $\oplus_j \mathcal{V}_{i_j} \to \mathcal{M}$ where the $i_j$ can be made as negative as you like. We pick a closed point $y \in Y$ and let $\mathcal{K} := \ker (\mathcal{M} \to \mathcal{M} \otimes k(y))$. Pick $i \ll 0$ such that $\Ext^1_Y(\mathcal{V}_i,\mathcal{K}) = 0$ so by projectivity, we may lift any epimorphism $\bar{\phi} \colon \mathcal{V}_i^r \to \mathcal{M} \otimes k(y)$ to a morphism $\phi \colon \mathcal{V}_i^r \to \mathcal{M}$. Now $\phi$ is surjective in a neighbourhood of $y$ by Nakayama's lemma so we are done by d\'evissage. 
\end{proof}

\begin{lemma}  \label{lem:twistBundleAmple}
Let $\mathcal{V}$ be a vector bundle of fixed positive rank on a projective scheme $Y$ and $\mathcal{L}$ an ample line bundle. Then the pair $(\mathcal{V}, (-) \otimes_Y \mathcal{L})$ is ample in the sense of \cite[(4.2.1)]{az}, that is, $(\mathcal{V}\otimes_Y \mathcal{L}^{\otimes i})$ is ample.
\end{lemma}
\begin{proof}
By Lemma~\ref{lem:amplenessForBundles}, we need only prove projectivity and so consider an exact sequence of coherent sheaves 
$$0 \to \mathcal{M}' \to \mathcal{M} \to \mathcal{M}'' \to 0.$$
Projectivity follows from 
$$\Ext^1_Y(\mathcal{V} \otimes \mathcal{L}^{\otimes -i}, \mathcal{M}') = H^1(Y, \mathcal{V}^* \otimes \mathcal{M}' \otimes \mathcal{L}^{\otimes i}) = 0$$
for $i \gg 0$. 
\end{proof}



\begin{corollary} \label{cor:P2noetherian}
Let $\underline{\mathcal{E}}$ be a planar helix.  Then
\begin{enumerate}
\item{} The $\mathbb{Z}$-algebra $\operatorname{End}({\underline{\mathcal{E}}})$ is noetherian,

\item{} The canonical map $\mathbb{S}^{nc}(\underline{\mathcal{E}}) \longrightarrow \operatorname{End}({\underline{\mathcal{E}}})$ is an isomorphism,

\item{} There are equivalences of categories 
$$
{\sf proj }\mathbb{S}^{nc}(\underline{\mathcal{E}}) \equiv {\sf Coh }\mathbb{P}^{2}
$$ 
and 
$$
{\sf Proj }\mathbb{S}^{nc}(\underline{\mathcal{E}}) \equiv {\sf Qcoh }\mathbb{P}^{2}.
$$
\end{enumerate}
\end{corollary}

\begin{proof}
To prove (1), we use Lemma~\ref{lem:twistBundleAmple} to see $\underline{\mathcal{E}}$ and its dual are ample so (1) now follows from Proposition~\ref{prop.leftright}.

Part (2) follows from the fact that $\operatorname{End}({\underline{\mathcal{E}}})$ is quadratic by \cite[Theorem 4.2]{bp}.  The first part of (3) now follows from (2) and the fact that $\underline{\mathcal{E}}$ is ample for ${\sf Coh }\mathbb{P}^{2}$.  To prove the second part of (3), we note that by (1) and \cite[Theorem 8.9]{pop}, the category ${\sf Proj }\mathbb{S}^{nc}(\underline{\mathcal{E}})$ can be recovered up to equivalence by its full subcategory of noetherian objects, ${\sf proj }\mathbb{S}^{nc}(\underline{\mathcal{E}})$, and similarly ${\sf Qcoh }\mathbb{P}^{2}$ can be recovered up to equivalence by ${\sf Coh }\mathbb{P}^{2}$.  
\end{proof}

We now show that planar helices restrict to elliptic helices.

\begin{proposition} \label{prop:restrictPlanarHelices}
Let $\underline{\mathcal{E}}$ be a planar helix and $X \subset \mathbb{P}^2$ be a smooth cubic. Then $\underline{\mathcal{E}}|_X := (\mathcal{E}_i|_X)$ is an elliptic helix on $X$. 
\end{proposition}
\begin{proof}
We first show that the restriction map induces isomorphisms $(\mathcal{E}_{i-1},\mathcal{E}_{i}) \simeq (\mathcal{E}_{i-1}|_X,\mathcal{E}_{i}|_X)$. Consider the exact sequence
\begin{equation}  \label{eq:restrictEtoX}
     0 \to \mathcal{E}_i(-3) \to \mathcal{E}_i \to \mathcal{E}_i|_X \to 0.
\end{equation}
We deduce the exact sequence
$$0 \to {}^{0}(\mathcal{E}_{i-1}, \mathcal{E}_i(-3)) \to {}^{0}(\mathcal{E}_{i-1}, \mathcal{E}_i) \to {}^{0}(\mathcal{E}_{i-1}, \mathcal{E}_i|_X) \to \ ^1(\mathcal{E}_{i-1}, \mathcal{E}_i(-3)).$$
The desired isomorphism follows from the fact that $\mathcal{E}_i(-3) \simeq \mathcal{E}_{i-3}$ and the fact that $\{\mathcal{E}_{i-3},\mathcal{E}_{i-2},\mathcal{E}_{i-1}\}$ is exceptional. 

Now $R_{\mathcal{E}_i}\mathcal{E}_{i-1}$ is a bundle so the right mutation sequence 
$$ 0 \to \mathcal{E}_{i-1} \to \ ^*(\mathcal{E}_{i-1},\mathcal{E}_i) \otimes \mathcal{E}_i \to R_{\mathcal{E}_i}\mathcal{E}_{i-1} \to 0$$
restricts to a right mutation sequence on $X$ and similarly for left mutations. It now follows that $\underline{\mathcal{E}}|_X$ is also a weak helix. 

Applying $\mathbf{R}\Hom(\mathcal{E}_i, -)$ to (\ref{eq:restrictEtoX}) similarly shows that $\End \mathcal{E}_i|_X = k$. Remark~\ref{rem:ellipticHelix}(4) completes the proof. 
\end{proof}

\section{On the numerics of three-periodic elliptic helices}  \label{sec.numbericalSetup}
We now investigate the numerical properties of three-periodic elliptic helices $\underline{\cE}$ on an elliptic curve $X$. A key invariant is a triple $\vec{\Delta}$ recording the dimensions of $\End \underline{\cE}$ in degree one. In this section, we analyse the numerics of helices to show only scaled Markov triples $\vec{\Delta}$ can give rise to indexed algebras of polynomial growth.  We enlarge our scope in this section to include objects in $D^b(X)$, as many of the computations work in this context. 

We have an isomorphism $(\deg,\rk) \colon K_0(X) \to \mathbb{Z}^2$. For objects $E,E' \in D^b(X)$ we let $[E],[E']$ denote their images in $K_0(X)$ and 
\begin{equation}   \label{eq:eulerFormOnX}
\chi([E],[E']) = \chi(E,E'):=  \chi(\text{R}\Hom(E,E')) = 
\begin{vmatrix}
    \deg E' & \deg E \\ \rk E' & \rk E
\end{vmatrix}.
\end{equation}
Let $E_0,E_1,E_2 \in D^b(X)$ and consider the right mutation $E'_2 := R_{E_1}E_0$ (defined in terms of exact triangles, see \cite[Proposition 7.1]{bp}) which has
$$[E'_2] = \chi(E_0,E_1) [E_1] - [E_0].
$$
If we write 
$$
\Delta_{ij} = \chi(E_i,E_j)
$$ 
then the iterated right mutation $E_3:= R_{E_2} E'_2$ satisfies
$$[E_3] = \chi(E'_2,E_2) [E_2] - [E'_2] = (\Delta_{01}\Delta_{12} - \Delta_{02})[E_2] - \Delta_{01}[E_1] + [E_0].
$$
Applying $\chi(E_2,-)$ to both sides we obtain the formulas
\begin{equation}  \label{eq.basicMutateFormulae}
\Delta_{23} = \Delta_{01}\Delta_{12} - \Delta_{02}, \ \ [E_3] = [E_0] - \Delta_{01} [E_1] + \Delta_{23} [E_2].
\end{equation}
We may repeat the process forwards by mutating $E_1$ past $E_2$ and $E_3$ to define $E_4$, and backwards by mutating $E_2$ left past $E_1$ and $E_0$ to get $E_{-1}$ and so on. If all these objects turn out to be vector bundles, then we obtain an elliptic helix of period three in ${\sf Coh} X$ as defined in Definition~\ref{def:ellipticHelix}. 

\begin{definition}  \label{def.helicalCOnd}
The sequences $\underline{E} = \{E_i\}_{i \in \mathbb{Z}}$ generated by a seed $(E_0,E_1,E_2)$ by mutation as above is said to satisfy the {\em three-helical condition}. We call $\underline{E}$ the {\em three-helix in $D^b(X)$ generated by the seed}.
\end{definition}

We abbreviate $\Delta_{i-1,i}=: \Delta_i$ and introduce 
$$
A_{j} := \begin{pmatrix} 0 & 1 & 0 \\ 0 & 0 & 1 \\ 1 & -\Delta_{j} & \Delta_{j+2} \\
\end{pmatrix}
$$
so (\ref{eq.basicMutateFormulae}) gives
\begin{equation}  \label{eq.matrixMutateFormula}
A_{i} \begin{pmatrix} [E_{i-1}] \\ [E_{i}] \\ [E_{i+1}] \end{pmatrix} =
\begin{pmatrix} [E_i] \\ [E_{i+1}] \\ [E_{i+2}] \end{pmatrix}.
\end{equation}
We let $A:=A_{3}A_{2}A_{1}$.

The following is \cite[Corollary~7.11]{morph} which is easily reproved with our current notation.
\begin{proposition}  \label{prop.3period}
We have $\Delta_{i+3} = \Delta_i$ for all $i$ and hence $A_{i+3} = A_i$. 
\end{proposition}
\begin{proof}
Applying $\chi(E_1,-)$ to (\ref{eq.basicMutateFormulae}) now gives
$$ - \Delta_1 + \Delta_2 \Delta_3 = \Delta_{13} = \Delta_2\Delta_3 - \Delta_4$$
so $\Delta_4 = \Delta_1$. 
\end{proof}

The following is elementary using the 3-periodicity of the $\Delta_i$'s.
\begin{proposition} \label{prop.eigenA}
$\det(A_i) = 1$ and $A_i$ cyclically permutes the coordinates of $\begin{pmatrix} \Delta_{i+1} \\ \Delta_{i+2} \\ \Delta_{i} \end{pmatrix}$ for any $i$. In particular $A:= A_3A_2A_1$ has determinant 1 and eigenvector $\begin{pmatrix} \Delta_{2} \\ \Delta_{3} \\ \Delta_{1} \end{pmatrix}$ with eigenvalue 1. 
\end{proposition}
As one can check by explicit calculation, we have 
\begin{equation}  \label{eq:Amatrix}
A = \begin{pmatrix}
1 & -\Delta_{1} & \Delta_{3} \\ \Delta_{1} & 1-\Delta_{1}^{2} & \Delta_{1}\Delta_{3}-\Delta_{2} \\ \Delta_{1}\Delta_{2}-\Delta_{3} & \Delta_{1}\Delta_{3}+\Delta_{2}(1-\Delta_{1}^{2}) & \Delta_{2}(\Delta_{1}\Delta_{3}-\Delta_{2})+1-\Delta_{3}^{2}    
\end{pmatrix}.    
\end{equation}

This shows the other eigenvalues of $A$ are
\begin{equation}  \label{eq.evaluesA}
\lambda_{\pm}:= \frac{1}{2}\biggl( 2+\Delta_{1}\Delta_{2}\Delta_{3}-(\Delta_{1}^{2}+\Delta_{2}^{2}+\Delta_{3}^{2}) \pm \sqrt{(2+\Delta_{1}\Delta_{2}\Delta_{3}-(\Delta_{1}^{2}+\Delta_{2}^{2}+\Delta_{3}^{2}))^{2}-4} \biggr).
\end{equation}

\begin{proposition} \label{prop.MarkovTypeCondition}
The eigenvalues of $A$ have modulus 1 iff 
$$\Delta_1^2 + \Delta_2^2 + \Delta_3^2 - \Delta_1\Delta_2\Delta_3 = 0,1,2,3 \ \text{or}\ 4.$$ 
In this case the eigenvalues are 1 and (a) 1,1 (b) $e^{\pm i\pi/3}$  (c) $\pm i$ (d) $e^{\pm i2\pi/3}$  (e) $-1,-1$ respectively. 
\end{proposition}
\begin{proof}
The sum of the remaining two eigenvalues of $A$ has to have magnitude $\leq 2$ since $\lambda_{+}+\lambda_{-}=2\cos \theta$ for some $\theta \in \mathbb{R}$, which gives the 5 possibilities above. 
\end{proof}

Calculating directly with the explicit form of $A$ given in (\ref{eq:Amatrix}) also gives the following.
\begin{lemma}  \label{lem.AminusI}
Suppose $\vec{\nabla} := \left(
\begin{smallmatrix}
\Delta_2 \\ 
\Delta_3 - \Delta_1 \Delta_2 \\ \Delta_1  
\end{smallmatrix}\right)$. Then 
\begin{enumerate}
\item{} $\vec{\nabla}^T (A-I) = 0$, and

\item{} if $\Delta_1 \neq 0$, then $\rk (A-I) =2$.
\end{enumerate}
\end{lemma}

We will be interested in the case where $\underline{\cE} = (\cE_i)_{i \in \mathbb{Z}}$ is an elliptic helix. In this case $\End(\underline{\cE})$ is a positively $\mathbb{Z}$-indexed algebra and for $i<j$ we have $\dim \Hom(\cE_i,\cE_j) = \chi(\cE_i,\cE_j)$. We are interested in the Hilbert series of $\End(\underline{\cE})$ and hence introduce the following
\begin{definition} \label{def.growth}
An elliptic helix $\underline{\cE}$ is said to have {\em polynomial growth} if there is a polynomial $p(n)$ such that for all $i<j$, we have $\chi(\cE_i,\cE_j) \leq p(j-i)$. If on the other hand, for any $i$ we have $\overline{\lim}_{n \to \infty} \chi(\cE_i,\cE_{i+n})^{1/n} > 1$, then we will say that $\underline{\cE}$ has {\em exponential growth}.
\end{definition}
Consider now an elliptic helix $\underline{\cE}$. We turn to studying the growth rates of the sequences $\deg(\cE_j), \rk(\cE_j)$. If $A$ is diagonalisable, say with eigenvalues $1,\lambda_+,\lambda_-$, then fixing $i = 0,1$ or 2, eigenvalue theory gives constants $\delta_1,\delta_+,\delta_-,\rho_1,\rho_+,\rho_-$such that 
\begin{equation}  \label{eq.degRkFormula}
\begin{pmatrix}
\deg(\cE_{i+3n}) \\ \rk(\cE_{i+3n})
\end{pmatrix} = 
\begin{pmatrix}
\delta_1 + \lambda_+^n\delta_+ + \lambda_-^n \delta_- \\
\rho_1 + \lambda_+^n\rho_+ + \lambda_-^n \rho_-
\end{pmatrix}.
\end{equation}
If we are neither in the scaled Markov case where $\Delta_1^2 + \Delta_2^2 + \Delta_3^2 = \Delta_1\Delta_2\Delta_3$ and so all eigenvalues of $A$ are 1, nor in case (e) of Proposition~\ref{prop.MarkovTypeCondition}, then the fact that $\lambda_+\lambda_- = 1$ ensures the eigenvalues $1,\lambda_+,\lambda_-$ are all distinct so $A$ is diagonalisable and (\ref{eq.degRkFormula}) applies. Otherwise, we have to use the theory of Jordan blocks. 
\begin{proposition}  \label{prop.whyMarkov}
Let $\underline{\cE}$ be an elliptic helix. Then either $\underline{\cE}$ has exponential  growth or we must be in the {\em scaled Markov case} where $\Delta_1^2 + \Delta_2^2 + \Delta_3^2 = \Delta_1\Delta_2\Delta_3$ in which case $\underline{\cE}$ has polynomial growth. 
\end{proposition}
\begin{proof}
Suppose first that $A$ is diagonalisable so we can use Equation~(\ref{eq.degRkFormula}) above. Let $j \in \{0,1,2\}$ and $d_j = \deg(\cE_j), r_j = \rk(\cE_j)$. Then  
\begin{equation}  \label{eq:chiEE}
\chi(\cE_j,\cE_{j+3n}) = 
\begin{vmatrix}
    \delta_1  & d_j \\
\rho_1   & r_j
\end{vmatrix} +
\lambda_+^n\begin{vmatrix}
    \delta_+ & d_j \\
\rho_+  & r_j
\end{vmatrix} +
\lambda_-^n\begin{vmatrix}
    \delta_- & d_j \\
\rho_-  & r_j
\end{vmatrix}.    
\end{equation}

Now $\chi(\cE_j,\cE_{j+3n})>0$ for $n>0$ and $\chi(\cE_j,\cE_{j+3n})<0$ for $n<0$. This forces the second and third determinants above to be non-zero. If $\lambda_{\pm}$ are real and not $\pm 1$, then we have 
exponential growth since we may assume $\lambda_+ >1$. Otherwise, by Proposition~\ref{prop.MarkovTypeCondition}, $\lambda_{\pm}$ are roots of unity so $\chi(\cE_j,\cE_{j+3n})$ is periodic in $n$ and we cannot have $\chi(\cE_j,\cE_{j+3n})>0$ for $n>0$ and $\chi(\cE_j,\cE_{j+3n})<0$ for $n<0$. 

Now suppose that $A$ is not diagonalisable. Assuming we are not in the scaled Markov case, then -1 must be an eigenvalue corresponding to a $2\times 2$ Jordan block for $A$. Repeating the above analysis we find this time that there are three determinants $D_1,D_{-1}, D'_{-1}$ such that 
$$\chi(\cE_j,\cE_{j+3n}) = D_1 + (-1)^n D_{-1} + (-1)^{n-1}nD'_{-1}.$$
Once again, it is not possible that $\chi(\cE_j,\cE_{j+3n})$ always has the same sign as $n$. 

Finally, suppose we are in the scaled Markov case so $A$ only has 1 as its sole eigenvalue. Using the appropriate Jordan canonical form for $A$, we find scalars $D,D',D''$ such that 
$\chi(\cE_i, \cE_{j+3n}) = D + n D' + n^2 D''$
so we are done. (In fact $D''$ must also be zero so the growth is linear). 
\end{proof}

\section{Hilbert series of regular algebras associated to elliptic helices}  \label{sec:hilbertSeries}

In this section, we continue our study of an elliptic helix $\underline{\mathcal{E}}$ on an elliptic curve $X$ and use the notation in Section~\ref{sec.numbericalSetup}. We will compute the Hilbert series of the algebras $\End \underline{\mathcal{E}}$ and its quadratic cover $\mathbb{S} = \mathbb{S}^{nc}(\underline{\mathcal{E}})$. The method here differs from the direct ad hoc calculation for the ``equigenerated'' case of elliptic helices in \cite[Section~8]{morph}. The method also gives a nice explanation of the relationship between the two Hilbert series since they are both computed from the corresponding helices. 

We work with relative Hilbert series as defined in \cite[Definition~8.4]{morph}, the relevant ones of which are the following. 
\begin{eqnarray}
    H_{\mathbb{S},n}(t) & := \sum_{i \in \mathbb{Z}} \dim_k (e_n \mathbb{S}^{nc}(\underline{\cE})e_{n + i}) t^i \\
    H_{\cE,n}(t) & := \sum_{i \in \mathbb{Z}} \dim \Hom_X(\cE_{-n-i},\cE_{-n})t^i.
\end{eqnarray}
The latter is the Hilbert series for $\End (\underline{\cE})$. 
We will employ the same method for computing both of these and start with $H_{\cE,n}$ first. 

We re-write the exact sequence (\ref{eq:helixExactSeq}) induced by the helical condition as
\begin{equation}  \label{eq:helixExactSeqNumerical}
0 \to \cE_{-n-3} \to \cE_{-n-2}^{\oplus \Delta_{-n-2}} \to \cE_{-n-1}^{\oplus \Delta_{-n}} \to \cE_{-n} \to 0. 
\end{equation}

Bilinearity of the Euler form gives
\begin{equation}  \label{eq:hilbertRecurRelation}
\chi(\cE_{-n-i},\cE_{-n}) -\Delta_{-n}\ \chi(\cE_{-n-i},\cE_{-n-1}) + \Delta_{-n-2}\ \chi(\cE_{-n-i},\cE_{-n-2}) -\chi(\cE_{-n-i},\cE_{-n-3}) =0
\end{equation}
which suggests the following result.

\begin{lemma}  \label{lem:hilbertRecurRelation}
$H_{\cE,n}(t) - \Delta_{-n} t H_{\cE,n+1}(t) +\Delta_{-n-2} t^2 H_{\cE,n+2}(t)- t^3 H_{\cE,n+3}(t) = 1-t^3$.
\end{lemma}
\begin{proof}
We just compare coefficients of $t^i$ on both sides. If $i >3$, then all the $\chi$ terms in (\ref{eq:hilbertRecurRelation}) are given by dimensions of Hom spaces and hence, coefficients of the relative Hilbert series on the left hand side. Thus (\ref{eq:hilbertRecurRelation}) immediately shows the power series on the left is a polynomial of degree $\leq 3$. The constant term on the left is also clearly 1 whilst the coefficient of $t$ is 
$$\dim (\cE_{-n-1},\cE_{-n}) - \Delta_{-n} \dim(\cE_{-n-1},\cE_{-n-1}) = 0.$$
For the coefficients of $t^2, t^3$, it is easier to ``correct'' (\ref{eq:hilbertRecurRelation}) with Ext computations. The co-efficient of $t^2$ is by Serre duality 
$$\Delta_{-n-2} \dim \Ext^1(\cE_{-n-2},\cE_{-n-2}) - \dim \Ext^1(\cE_{-n-2},\cE_{-n-3}) = 0$$
and that of $t^3$ is $-\dim \Ext^1(\cE_{-n-3},\cE_{-n-3}) = -1$. 
\end{proof}

In the ``equigenerated'' elliptic helices of \cite[Section~8]{morph} there was a single Hilbert series whereas here, we have three as per the following. 
\begin{proposition}  \label{prop:hilbertSeries3periodic}
We have $H_{\cE,n+3}(t) = H_{\cE,n}(t)$.
\end{proposition}
\begin{proof}
We prove 3-periodicity simultaneously for all $n$ by induction on the degree of the coefficients. The coefficients of negative degree are all zero. For non-negative degrees, Lemma~\ref{lem:hilbertRecurRelation} expresses any term of $H_{\cE,n}(t)$ in terms of the $\Delta_n$, which are 3-periodic in $n$ by Proposition~\ref{prop.3period}, and lower order terms in other $H_{\cE,m}(t)$ which, by induction are also 3-periodic. 
\end{proof}

\begin{theorem}  \label{thm:hilbertSeries}
Let 
\begin{equation}  \label{eq:denominatorHilbert}
D(t) := 
\begin{pmatrix}
1 - t^3 & - \Delta_3 t & \Delta_1 t^2 \\
\Delta_3 t^2 & 1 - t^3 & - \Delta_2 t \\
- \Delta_1 t & \Delta_2 t^2 & 1 - t^3
\end{pmatrix}.
\end{equation}
Then 
\begin{equation}  \label{eq:hilbertSeries}
\begin{pmatrix}
    H_{\mathbb{S},0}(t) \\ H_{\mathbb{S},1}(t) \\ H_{\mathbb{S},2}(t) 
\end{pmatrix}
 = 
 \frac{1}{D(t)}
 \begin{pmatrix}
     1 \\ 1 \\ 1
 \end{pmatrix}
\ \text{ and } \ 
\begin{pmatrix}
    H_{\cE,0}(t) \\ H_{\cE,1}(t) \\ H_{\cE,2}(t) 
\end{pmatrix}
 = 
 \frac{1-t^3}{D(t)}
 \begin{pmatrix}
     1 \\ 1 \\ 1
 \end{pmatrix}.
\end{equation}
Furthermore, $H_{\mathbb{S},n+3}(t) = H_{\mathbb{S},n}(t)$ for all $n$.
\end{theorem}
\begin{proof}
The Hilbert series for $\End({\underline{\cE}})$ is just a re-statement of Lemma~\ref{lem:hilbertRecurRelation} taking into account the 3-periodicity of Proposition~\ref{prop:hilbertSeries3periodic}. To obtain the Hilbert series for $\mathbb{S}^{nc}(\underline{\cE})$ we use the Koszul resolution guaranteed by Theorem~\ref{thm:BPBisKoszul}. Abbreviating $\mathbb{S}$ for $\mathbb{S}^{nc}(\underline{\cE)}$, this is  
\begin{equation}
    0 \rightarrow e_n\mathbb{S} \rightarrow e_{n+1}\mathbb{S}^{\oplus \Delta_{-n-2}} \rightarrow e_{n+2}\mathbb{S}^{\oplus \Delta_{-n}} \rightarrow e_{n+3}\mathbb{S} \rightarrow \mathbb{S}_{n+3,n+3} \rightarrow 0.
\end{equation}
Taking alternating sums of dimensions in degree $n+j$ gives the following analogue of Lemma~\ref{lem:hilbertRecurRelation}:
\begin{equation*}
H_{\mathbb{S},n}(t) - \Delta_{-n} t H_{\mathbb{S},n+1}(t) +\Delta_{-n-2} t^2 H_{\mathbb{S},n+2}(t)- t^3 H_{\mathbb{S},n+3}(t) = 1.
\end{equation*}
Now 3-periodicity of $H_{\mathbb{S},n}(t)$ follows as before as does the formula for the Hilbert series. 
\end{proof}

We do not know if in general, noetherian $\mathbb{Z}$-algebras have ``sub-exponential'' growth as is the case for graded algebras by \cite{StephensonZhang}. However, in our case we do have the following. 

\begin{proposition}  \label{prop:noetherianIsMarkov}
Let $\underline{\cE}$ be an elliptic helix such that $\End (\underline{\cE})$ is right noetherian. Then $\underline{\cE}$ has polynomial growth and we are in the scaled Markov case. 
\end{proposition}
\begin{proof}
By Proposition~\ref{prop.whyMarkov}, it suffices to assume $\underline{\cE}$ has exponential growth and construct a non-finitely generated submodule of $e_0 B$ where $B:=\End (\underline{\cE})$. We can use Stephenson-Zhang's argument \cite[Theorem~1.2]{StephensonZhang} as follows. From (\ref{eq:chiEE}) in the proof of Proposition~\ref{prop.whyMarkov}, we see that $d_n:= \dim e_0 B e_{3n}$ grows exponentially so by \cite[Lemma~1.1.3]{StephensonZhang}, there is a sequence $l_1<l_2 < \ldots$ such that 
\begin{equation}  \label{eq:SZdSequence}
d_{l_s} >\sum_{i<s} d_{l_s-l_i}    
\end{equation}
for all $s\in \mathbb{N}$. We show that one can find, inductively on $s$, elements $x_s \in B_{0,3l_s}$ such that $x_s \notin x_1B + \ldots + x_{s-1} B$. By the 3-periodicity of Hilbert series in Proposition~\ref{prop:hilbertSeries3periodic}, we know that $\dim_k x_i B e_{3l_s} \leq \dim_k e_{3l_i}Be_{3l_s} = d_{l_s-l_i}$. Thus our desired $x_s$ exists by (\ref{eq:SZdSequence}). 
\end{proof}

As in Artin-Tate-van den Bergh's classic study \cite{atv} of regular graded algebras of dimension three, the relationship between the Hilbert series is a manifestation of the fact that we have an embedding of a noncommutative elliptic curve as a cubic in a noncommutative projective plane. 
\begin{corollary}  \label{cor:cubicDivisor}
Let $\underline{\mathcal{E}}$ be an elliptic helix and let $\phi: \mathbb{S}^{nc}(\underline{\mathcal{E}}) \rightarrow \operatorname{End }(\underline{\mathcal{E}})$ denote the canonical surjection.  Then $\operatorname{ker }\phi$ is generated by a normal family of regular elements of degree three.
\end{corollary}
\begin{proof}
This is a generalisation of \cite[Theorem~8.7]{morph} to the not necessarily equigenerated case. In fact, the proof there carries over verbatim once we realise that the equigeneration hypothesis was only used to ensure $H_{\cE,n}(t) = (1-t^3) H_{\mathbb{S},n}(t)$. This latter is just Theorem~\ref{thm:hilbertSeries}. 
\end{proof}

\section{Constructing seeds with given $\Delta_i$'s: the scaled Markov case}  \label{sec:seedsMarkovCase}

We wish to construct all elliptic helices $\underline{\cE}$ on an elliptic curve $X$ with polynomial growth. Towards this end, in this section, we construct seeds with given invariants $\Delta_i$, in the case that $(\Delta_{2}, \Delta_{3}, \Delta_{1})$ form a scaled Markov triple. 

We will utilize numerous properties of Markov triples, which we briefly review now.  We will often use the fact that Markov triples are relatively prime \cite[Corollary 3.4]{hundred}.  Another result that will be useful to us is that Markov triples can be {\it mutated} to form new triples \cite[Section 3.1]{hundred}:  if $(m_{1}, m_{2}, m_{3})$ is a Markov triple, then so are
\begin{equation} \label{eqn.markovtrip2}
(3m_{2}m_{3}-m_{1}, m_{2}, m_{3}), (m_{1}, 3m_{1}m_{3}-m_{2}, m_{3}), (m_{1}, m_{2}, 3m_{1}m_{2}-m_{3}).
\end{equation}

\begin{lemma} \label{lemma.three}
If $m$ is a Markov number (i.e. a component of a Markov triple), then $m$ isn't divisible by three.
\end{lemma}

\begin{proof}
We know from \cite[Lemma 5, p. 28]{cassels} that any Markov triple can be obtained from the Markov triple $(1,1,1)$ by a sequence of Markov mutations. However, from the formula (\ref{eqn.markovtrip2}), we see that modulo 3, such a mutation only changes the sign of one of the components. 
\end{proof}

We now fix a triple of positive integers 
$\Delta_1, \Delta_2, \Delta_3$ which satisfies the scaled Markov condition $\Delta_1^2 + \Delta_2^2 + \Delta_3^2  = \Delta_1\Delta_2\Delta_3$ from (\ref{eqn.markovtrip}). Note then that $\Delta_1,\Delta_2,\Delta_3$ are all multiples of 3 and so dividing by 3 we get one of the usual Markov triples. We now ask the converse question, when and how can we construct a {\em seed} $[E_i] = (d_i,r_i) \in K_0(X) = \mathbb{Z}^2$ for $i=0,1,2$ such that $\chi(E_{i-1},E_i) = \Delta_i$ for $i=1,2,3$ where $E_3$ is the double right mutation of $E_0$ defined in Section~\ref{sec.numbericalSetup}. We also want the seed to generate an elliptic helix $\underline{\cE}$ of bundles on $X$, so in particular, we insist that the ranks $r_i$ of the $E_i$ are positive integers. In this section, we give a complete answer. 

If we write $\vec{d} = \begin{pmatrix}
d_0 \\ d_1 \\ d_2 \end{pmatrix}$, and $\vec{r} = \begin{pmatrix} r_0 \\ r_1 \\ r_2 
\end{pmatrix}$, then (\ref{eq.basicMutateFormulae}) gives the following.
\begin{proposition} \label{prop.crossProductCond}
With the above notation, we have $\chi(E_{i-1},E_i) = \Delta_i$ for $i=1,2,3$ if and only if 
\begin{equation} \label{eq.crossProductCondition}
\vec{d} \times \vec{r} = 
-\begin{pmatrix}
 \Delta_2 \\  - \Delta_{02} \\ \Delta_1
\end{pmatrix}    
\end{equation}
where $\Delta_{02} = \Delta_1\Delta_2 - \Delta_3$.
\end{proposition}
Inspired by (\ref{eqn.markovtrip2}), we call $\vec{
\Delta} = (\Delta_2,\Delta_3, \Delta_1) \mapsto (\Delta_2, \Delta_1\Delta_2 - \Delta_3,\Delta_1)$ a {\it scaled Markov mutation}. 

Now, let $A$ be the matrix defined in Proposition~\ref{prop.eigenA} so the ranks of the helix components $E_i$ are given by $A^n \vec{r}, n \in \mathbb{Z}$. 

\begin{proposition} \label{prop:risdelta}
Suppose that the seed given by $\vec{d}, \vec{r}$ satisfies Equation~(\ref{eq.crossProductCondition}) and $A^n \vec{r}$ consists of positive integers for all $n$. Then 
\begin{equation} \label{eq.rMarkov}
\vec{\Delta}:=
\begin{pmatrix} \Delta_{2} \\ \Delta_{3} \\ \Delta_{1}
\end{pmatrix} = 3 \vec{r}  \ \ \text{or } \vec{r}
\end{equation}
If $\vec{\Delta}=3 \vec{r}$, then $\vec{r}$ is a Markov triple, that is, 
\begin{equation}  \label{eq.Markov}
r_0^2 + r_1^2 + r_2^2 = 3 r_0r_1r_2,   
\end{equation}
and we have
\begin{equation}  \label{eq:AminusId}
A^n \vec{d} = \vec{d} + 3n \vec{\Delta}.
\end{equation}
\end{proposition}
\begin{proof}
From Proposition~\ref{prop.MarkovTypeCondition} we know $A$ has eigenvalue 1 with multiplicity 3. Furthermore, by Lemma~\ref{lem.AminusI}(2), $\ker (A-I) = \mathbb{Q} \vec{\Delta}$ so $A$ must consist of a single Jordan block. Thus $\ker(A-I)^2 = \text{im} (A-I)$ is 2-dimensional.  We now claim $\ker(A-I)^2 = \vec{\nabla}^{\perp}$ where $\vec{\nabla} = \left(
\begin{smallmatrix}
\Delta_2 \\ 
\Delta_3 - \Delta_1 \Delta_2 \\ \Delta_1  
\end{smallmatrix}\right)$.  To prove the claim, we observe that since $v \in \ker(A-I)^{2}=\text{im}(A-I)$ implies $v=(A-I)w$, $\vec{\nabla} \cdot v = 0$ by Lemma~\ref{lem.AminusI}(1).

Thus for Equation~(\ref{eq.crossProductCondition}) to hold, we must have $\vec{r} \in \ker (A-I)^2$, and hence in $\text{im} (A-I)$ so $A \vec{r} = \vec{r} + \alpha \vec{\Delta}$ for some $\alpha \in \mathbb{Q}$. But then $A^n \vec{r} = \vec{r}  + n\alpha \vec{\Delta}$ which must have negative entries for some $n$ unless $\alpha = 0$. We have thus shown that $\vec{r} \in \ker(A-I) = \mathbb{Q} \Delta$. Now the greatest common divisor of all the coordinates of $\vec{\nabla}$ is 3, so either $\vec{r} = \vec{\Delta}$ or $\frac{1}{3}\vec{\Delta}$. 

Now suppose $\vec{r}=\frac{1}{3}\vec{\Delta}$.  Prompted by the first column of (\ref{eq:Amatrix}), we let 
$\vec{\delta}:= \frac{3}{\Delta_2} \begin{pmatrix}
0 \\ \Delta_1 \\ \Delta_1 \Delta_2 - \Delta_3    
\end{pmatrix}$. A direct computation shows that $\vec{\delta} \times \vec{r} = - \vec{\nabla}$ so $\vec{d} = \vec{\delta} + \beta \vec{\Delta}$ for some $\beta \in \mathbb{Q}$. It follows by direct computation that 
$$(A-I) \vec{d} = (A-I)\vec{\delta} = 3\vec{\Delta}$$
giving (\ref{eq:AminusId}).
\end{proof}
\begin{remark}
Recall from Proposition~\ref{prop.eigenA} that if Equation~(\ref{eq.rMarkov}) holds, then $A \vec{r} = \vec{r}$ so positivity of ranks is guaranteed.
\end{remark}

Consider the Markov triple $\vec{\rho} = (\rho_0,\rho_1,\rho_2)^T = \frac{1}{3}\vec{\Delta}$. We define three ``signed Markov'' mutations as follows. For any vector $\vec{v} = (v_0,v_1,v_2)^T\in \mathbb{Q}^3$, let
$$  
\mu_0(\vec{v}) = 
\begin{pmatrix}
v_0 + 3v_1v_2 \\ v_1 \\ v_2 
\end{pmatrix}, \quad
\mu_1(\vec{v}) = 
\begin{pmatrix}
v_0 \\ v_1 - 3v_0v_2\\ v_2 
\end{pmatrix}, \quad
\mu_2(\vec{v}) = 
\begin{pmatrix}
 v_0 \\ v_1 \\v_2 + 3v_0v_1
\end{pmatrix}, \quad
$$
Recall that the absolute values of the coordinates of $\mu_1(\vec{\rho})$ are  Markov triple since $\vec{\rho}$ is. 
The choice of signs means that the Markov equation (\ref{eq.Markov}) gives the conditions
\begin{equation}  \label{eq.MarkovMutatePerp}
\vec{\rho} \perp \mu_1(\vec{\rho})=:\vec{\mu}_1, \quad \mu_1(\vec{\rho}) \perp \mu_0(\mu_1(\vec{\rho}))=:\vec{\mu}_0, \quad \mu_1(\vec{\rho}) \perp \mu_2(\mu_1(\vec{\rho}))=:\vec{\mu}_2.
\end{equation}

We now consider the two cases in Proposition~\ref{prop:risdelta} separately:
\begin{enumerate}
\item{Case 1:}  $\vec{r} = \vec{\rho}$.  We solve Equation~(\ref{eq.crossProductCondition}) for $\vec{d}$ which now becomes 
\begin{equation} \label{eqn.case1}
    \vec{d} \times \vec{\rho} = -3 \vec{\mu}_1.
\end{equation}
Any solution $\vec{d}$, if it exists, is unique up to an additive multiple of $\vec{r}$. Also, $\vec{d}$ must be orthogonal to $\vec{\mu}_1$ so natural choices are linear combinations of $\vec{\mu}_0, \vec{\mu}_2$. One computes 
\begin{equation}  \label{eq.mumuCrossr}
\vec{\mu}_0 \times \vec{\rho} = -3r_2^2 \vec{\mu}_1, \quad \vec{\mu}_2 \times \vec{\rho} = -3r_0^2 \mu_1(\vec{\rho}).
\end{equation}

\item{Case 2:} $\vec{r} = 3 \vec{\rho}$.  A similar analysis as in Case 1 yields 
\begin{equation} \label{eqn.case2}
    \vec{d} \times \vec{\rho} = - \vec{\mu}_1
\end{equation}
and 
\begin{equation}  \label{eq.mumuCrossr2}
\vec{\mu}_0 \times \vec{\rho} = -3\rho_2^2 \vec{\mu}_1, \quad \vec{\mu}_2 \times \vec{\rho} = -3\rho_0^2 \mu_1(\vec{\rho}).
\end{equation}

\end{enumerate}

Recall that the coordinates of a Markov triple are relatively prime. Thus we have the following. 
\begin{theorem}  \label{thm.existMarkov}
Let $\vec{r} \in \mathbb{N}^3$ be a Markov triple and set $\vec{\Delta}:=(\Delta_2,\Delta_3,\Delta_1) = 3(r_0,r_1,r_2)$. Then one can find degrees $\vec{d} \in \mathbb{Z}^3$ such that the seed in $K_0(X)$ defined by $\vec{d}, \vec{r}$ has determinants given by $\vec{\Delta}$. In fact, given any integers $a,b$ such that $ar_0^2 + br_2^2 = 1$ we have the solution $\vec{d} = a \vec{\mu}_2 + b \vec{\mu}_0$ and all other solutions can be obtained from this by adding appropriate multiples of $\vec{r}$. 
\end{theorem}
\begin{proof}
It remains only to check $\gcd(d_i,r_i) = 1$. But any non-trivial common divisor of $d_i,r_i$ would also divide two of the components of $\mu_1(\vec{r})$ by (\ref{eqn.case2}). This contradicts the fact that the absolute values of the components of $\mu_1(\vec{r})$ are again a Markov triple. 
\end{proof}

It remains only to eliminate the possibility $\vec{r} = 3 \vec{\rho}$. 

\begin{theorem}  \label{thm:onlyMarkov}
Let $(E_0,E_1,E_2)$ be a seed that generates an elliptic helix $\underline{\mathcal{E}}$ on $X$ with polynomial growth. Then 
$$
(\rk E_0, \rk E_1,\rk E_2) =: \vec{r}
$$ 
is a Markov triple and $\vec{\Delta} = 3 \vec{r}$.
\end{theorem}
\begin{proof}
We know from Proposition~\ref{prop.whyMarkov} that $\vec{\Delta} = 3 \vec{\rho}$ for some Markov triple $\vec{\rho}$. Furthermore, Proposition~\ref{prop:risdelta} reduces the possibilities to $\vec{r} = \vec{\rho}$ or $3 \vec{\rho}$. We seek a contradiction to the latter case so assume $\vec{r} = 3 \vec{\rho}$. We solve now (\ref{eqn.case2}).  Since $\vec{\mu}_0,\vec{\mu}_2$ are linearly independent, as one can check by explicit computation, we can find integers $b_0,b_2,m$ such that $\vec{d} = \frac{b_0}{m}\vec{\mu}_0 + \frac{b_2}{m}\vec{\mu}_2$. We may of course, assume that $m$ is as small a positive integer as possible. From the computations in Equation~(\ref{eq.mumuCrossr2}), this reduces to 
\begin{equation}  \label{eq.mandbs}
m = 3(b_0 \rho_2^2 + b_2 \rho_0^2 )
\end{equation}
so in particular, $m$ is divisible by 3. 
Writing $\rho_{02}$ for $\rho_1 - 3 \rho_0\rho_2$ we have 
\begin{equation}  \label{eq:mus}
\vec{\mu}_1 =
\begin{pmatrix}
\rho_0 \\ \rho_{02} \\ \rho_2
\end{pmatrix}, \quad 
\vec{\mu}_0 =
\begin{pmatrix}
\rho_0 + 3 \rho_{02} \rho_2\\ \rho_{02} \\ \rho_2
\end{pmatrix}, \quad
\vec{\mu}_2 =
\begin{pmatrix}
\rho_0 \\ \rho_{02} \\ \rho_2 + 3 \rho_{02} \rho_0
\end{pmatrix}
\end{equation}
We now use integrality of the coordinates of $\vec{d}$. From the middle coordinate, we know that $(b_0 + b_2)\rho_{02} \in m \mathbb{Z}$. Now $-\rho_{02}$ is a Markov number so cannot be divisible by 3 by Lemma \ref{lemma.three}. Thus $3| b_0 + b_2$. Furthermore, $\rho_2^2 \equiv \rho_0^2 \equiv 1$ modulo 3 so also $3 | b_0\rho_2^2 + b_2 \rho_0^2$. From Equation~(\ref{eq.mandbs}), we thus see $9 | m$ which in turn shows $9| b_0 + b_2$. 


We next claim that $9\nmid 3b_0 \rho_2\rho_{02}$.  To prove the claim, we note that, by minimality of $m$, we may assume that $\mbox{gcd}(b_{0}, b_{2}, m)=1$.  Since $3|m$ and $3|b_{0}+b_{2}$, we must have $3 \nmid b_{0}$, since otherwise, $3$ would divide $b_{0}$, $b_{2}$ and $m$.  The claim thus follows.

Now
\begin{equation*}
d_0 = \frac{(b_0+b_2)\rho_0 + 3b_0\rho_2 \rho_{02}}{m} \notin \mathbb{Z},
\end{equation*}
yielding a contradiction which completes the proof.
\end{proof}

\section{Elliptic helices of Markov type are ample}  \label{sec:ellipticHelixAmple}
The goal of this section is to show that three-periodic elliptic helices on $X$ with polynomial growth are ample.  As a consequence, we prove, in Corollary \ref{cor.leftrightnoetherian}, that endomorphism algebras and their quadratic covers are noetherian.  In what follows, we refer to elliptic helices such that $(\Delta_{1}, \Delta_{2}, \Delta_{3})$ is a scaled Markov triple as {\it elliptic helices of Markov type}.

\begin{proposition} \label{prop.ellipticample}
Suppose $\underline{\mathcal{E}} = (\mathcal{E}_{i})_{i \in \mathbb{Z}}$ is an elliptic helix of Markov type. Let $d_{i}=\operatorname{deg }\mathcal{E}_{i}$, let $r_{i} = \operatorname{rank }\mathcal{E}_{i}$, and suppose $\mu_{i} := d_{i}/r_{i}$.  Then $\lim_{n \to -\infty} \mu_{n}=-\infty$, and $\underline{\mathcal{E}}$ is ample for ${\sf Coh }X$.    
\end{proposition}
\begin{proof}
We continue the notation from Section~\ref{sec:seedsMarkovCase}. From Proposition~\ref{prop:risdelta} we know that $\vec{r}$ is also a 1-eigenvector for $A$ so $r_n$ is 3-periodic in $n$ and in particular bounded. Also Proposition~\ref{prop:risdelta} (\ref{eq:AminusId}) shows that $d_n \to -\infty$ as $n \to - \infty$ so $\lim_{n \to -\infty} \mu_{n}=-\infty$. We now use the ampleness criterion of Lemma~\ref{lem:amplenessForBundles} and so consider $\mathcal{M} \in {\sf Coh} X$ and show $\Ext^1_X( \mathcal{E}_i, \mathcal{M}) = 0$ for $i \ll 0$. Filtering $\mathcal{M}$, we may assume by induction that $\mathcal{M}$ is a simple sheaf with slope $\mu$. Then by Serre duality we have 
$$ \Ext^1_X( \mathcal{E}_i, \mathcal{M}) \simeq \Hom_X(\mathcal{M}, \mathcal{E}_i)^* = 0$$
as soon as the slope of $\mathcal{E}_i$ drops below $\mu$. 
\end{proof}

\begin{corollary} \label{cor.leftrightnoetherian}
If $\underline{\mathcal{E}}$ is an elliptic helix of Markov type, then $\operatorname{End }(\underline{\mathcal{E}})$ and $\mathbb{S}^{nc}(\underline{\mathcal{E}})$ are noetherian.  
\end{corollary}

\begin{proof}
For the first assertion, it remains to check the left noetherian condition, and this will follow from Proposition \ref{prop.leftright} provided we can prove the dual helix is ample.  By Proposition \ref{prop.ellipticample}, it thus suffices to show the dual helix is an elliptic helix of Markov type.  Since every triple of components of the dual helix has ranks a Markov triple, this is clear in light of the proof of \cite[Theorem 7.23]{morph}.

The fact that $\mathbb{S}^{nc}(\underline{\mathcal{E}})$ is right noetherian follows immediately from Corollary \ref{cor:cubicDivisor} and Proposition \ref{prop.hbt}.  To prove it is left noetherian, it suffices to prove that the opposite algebra is right noetherian.  However, we know the opposite algebra of $\operatorname{End }(\underline{\mathcal{E}})$ is right noetherian by the first part of the proof.  Thus, the result holds by considering the short exact sequence from Corollary \ref{cor:cubicDivisor} applied to opposite algebras and again invoking Proposition \ref{prop.hbt}. 
\end{proof}

Let $\underline{\cE}$ be an elliptic helix of Markov type and $d_i, r_i$ denote the degree and rank of $\cE_i$ as usual. Here we give a more direct computation of the Hilbert series of $\End(\underline{\cE})$. Let $i',j' \in \{0,1,2\}$ and $m \in \mathbb{N}$. From Proposition~\ref{prop:risdelta} (\ref{eq:AminusId}) we have $d_{i'+3m} = d_{i'} + 3m \Delta_{i'-1}$. Now (\ref{eq:chiEE}) and the fact that $3r_j = \Delta_{j-1}$ shows for $j' < i' + 3m$ that 
\begin{equation*}
\dim_k \Hom(\cE_{j'}, \cE_{i' + 3m}) = 
\begin{vmatrix}
    d_{i'} + 3m \Delta_{i'-1} & d_{j'} \\
    r_{i'} & r_{j'}
\end{vmatrix}
 = \chi({\cE_{j'}, \cE_{i'})} + m \Delta_{i'-1}\Delta_{j'-1}.
\end{equation*}

\section{Mutating elliptic helices} \label{section.mutate}

Given an elliptic helix $\underline{\cE}$ on an elliptic curve $X$, we construct a new elliptic helix by mutating every third bundle. We show that in the Markov case, the triple of ranks of bundles changes correspondingly by a Markov mutation. There is a corresponding mutation of the quadratic cover which allows us to show they are proj equivalent to an elliptic quantum plane.

\begin{proposition}  \label{prop:mutateHelix}
Let $\underline{\cE}$ be an elliptic helix of bundles on $X$.
\begin{enumerate}
    \item Right  (or left) mutating every third bundle yields another elliptic helix $\underline{\cE'}$ so in particular, the following is an elliptic helix
    \begin{equation}  \label{eq:mutatedHelix}
        \ldots, \cE_1, R_{\cE_1} \cE_0, \cE_2, \cE_4, R_{\cE_4}\cE_3, \cE_5, \ldots
    \end{equation}
    \item Suppose we are in the Markov case of Theorem~\ref{thm:onlyMarkov} so in particular, the rank triple of the seed $\vec{r} := (r_0 = \rk \cE_0, r_1 = \rk \cE_1,r_2 = \rk \cE_2)$ is a Markov triple. Then the seed of (\ref{eq:mutatedHelix}) has rank triple $(r_1, 3r_2r_1 - r_0,r_2)$
\end{enumerate}
\end{proposition}
\begin{proof}
Note first that shifting the indices of an elliptic helix by an additive constant yields another elliptic helix, so part~(1) will be proved if we can show that the collection of stable bundles in (\ref{eq:mutatedHelix}) is an elliptic helix. First note that stable bundles $\cE$ on $X$ are spherical objects in the sense of \cite[Definition~8.1]{Huyb} so the right and left mutations $R_{\cE}, L_{\cE}$ extend to inverse auto-equivalences of $D^b(X)$, by \cite[Proposition~8.6]{Huyb}. Denoting these extensions also by $R_{\cE}, L_{\cE}$, \cite[Lemma~8.21]{Huyb} readily gives the following useful formulae for stable bundles $\cE,\cF$:
\begin{equation} \label{eq:HuybrechtsCommuteSpherical}
R_{\cF} \circ R_{L_{\cF}\cE} = R_{\cE} \circ R_{\cF}
, \quad 
R_{\cF} \circ R_{\cE} = 
R_{R_{\cF}\cE} \circ R_{\cF}.
\end{equation}
We first show that doubly right mutating the $i$-th bundle $\cE'_i$ in (\ref{eq:mutatedHelix}) gives the $i+3$-rd bundle $\cE'_{i+3}$. By symmetry, it suffices to do this for $i = 0,1,2$. Since $\cE$ is an elliptic helix we have
$$ R_{\cE_4} R_{\cE_2} (R_{\cE_1}\cE_0) = R_{\cE_4} \cE_3
$$
as desired. Using (\ref{eq:HuybrechtsCommuteSpherical}) we find
$$ R_{R_{\cE_4} \cE_3}  R_{\cE_4}\cE_2 = R_{\cE_4} R_{\cE_3} \cE_2 = \cE_5.
$$
Finally, note that since $\underline{\cE}$ is an elliptic helix, $R_{\cE_1}\cE_0 = L_{\cE_2}\cE_3$ hence
$$ R_{\cE_2} R_{L_{\cE_2}\cE_3} \cE_1 = R_{\cE_3} R_{\cE_2} \cE_1 = \cE_4
$$
proving part~(1). 

For part~(2), just note that 
$$\rk R_{\cE_1}\cE_0 = \Delta_1r_1 - r_0 = 3r_2 r_1 - r_0.$$
\end{proof}

We now consider a corresponding mutation of the helix associated to an AS-regular Koszul algebra of global dimension three.

\begin{proposition}  \label{prop:mutateHelixOnP2}
Let $A$ be a Koszul AS-regular $\mathbb{Z}$-algebra of global dimension three and $\underline{\cO} = \{\cO_i:= \pi(e_{-i}A)\}$ be the associated helix in ${\sf Proj}(A)$. Right  (or left) mutating every third bundle yields another helix $\underline{\cO'}$ so in particular, the following is a helix
    \begin{equation}  \label{eq:mutatedHelixP2}
        \ldots, \cO_1, R_{\cO_1} \cO_0, \cO_2, \cO_4, R_{\cO_4}\cO_3, \cO_5, \ldots
    \end{equation}
\end{proposition}
\begin{proof}
We check the helical condition by mutating objects in (\ref{eq:mutatedHelixP2}) to the left twice. The proof is similar to that of Proposition~\ref{prop:mutateHelix} with one essential change. Write $\cO'_1 := R_{\cO_1} \cO_0 = L_{\cO_2} \cO_3$. Note first that $L_{\cO'_1} \circ L_{\cO_2} = L_{\cO_2} \circ L_{\cO_3}$ since, by \cite[Theorem~3.2]{bondal}, they are given by $j^![-2]$ where $j^!$ is the left adjoint to the inclusion functor $j$ of $\langle \cO_2,\cO_3 \rangle^{\perp} = \langle L_{\cO_2}\cO_3,\cO_2 \rangle^{\perp}$ into $D^b(A)$. It follows that 
$$L_{\cO'_1} \circ L_{\cO_2} \cO_4 = L_{\cO_2} \circ L_{\cO_3} \cO_4 = \cO_1$$
as desired. Using similar or easier methods, one verifies $L_{\cO_1}\circ L_{\cO'_1} \cO_2 = \cO_{-1}$ and $L_{\cO_{-1}}\circ L_{\cO_1} \cO'_1 = \cO_{-2}$. 
\end{proof}

\begin{proposition} \label{prop:mutationsCorrespond}
Let $\underline{\cE}$ be an elliptic helix and $\underline{\cE}'$ be the elliptic helix in (\ref{eq:mutatedHelix}) obtained by right mutating every third bundle. Let $\underline{\cO}$ be the helix associated to the quadratic part $\mathbb{S}^{nc}(\underline{\cE})$ of the endomorphism algebra $\End(\underline{\cE})$. Then the helix associated to $\mathbb{S}^{nc}(\underline{\cE}')$ is the mutated helix $\underline{\cO'}$ given in (\ref{eq:mutatedHelixP2}). 
\end{proposition}
\begin{proof}
We prove that $\End(\underline{\cO'})$ and $\End(\underline{\cE'})$ are canonically identified in degrees $\leq 2$ using the fact that the same is true for $\End(\underline{\cO})$ and $\End(\underline{\cE})$.
We consider degree 1 elements first. First note $\Hom(\cE_{-1},\cE_1) = \Hom(\cO_{-1},\cO_1)$ since they are degree two terms from the original helices $\underline{\cE}, \underline{\cO}$. Consider the exact sequence defining  right mutation
\begin{equation}  \label{eq:ROzero}
    0 \to \cO_0 \to \Hom(\cO_0,\cO_1)^* \otimes \cO_1 \to R_{\cO_1} \cO_0 \to 0
\end{equation}
which can be identified also with the sequence obtained from left mutating $R_{\cO_1} \cO_0$ past $\cO_1$. There is a similar sequence corresponding to $R_{\cE_1} \cE_0$ which together show that 
\begin{equation}  \label{eq:HomO1toROzero}
\Hom(\cE_1, R_{\cE_1} \cE_0) = \Hom(\cE_0,\cE_1)^* = \Hom(\cO_0,\cO_1)^* = \Hom(\cO_1, R_{\cO_1} \cO_0).
\end{equation}
Finally, the equality  $\Hom( R_{\cE_1} \cE_0,\cE_2)=\Hom(R_{\cO_1} \cO_0,\cO_2)$ follows similarly but now using the the left mutation sequence
\begin{equation}  \label{eq:LO3}
    0 \to  L_{\cO_2} \cO_3 \to \Hom(\cO_2,\cO_3) \otimes \cO_2 \to \cO_3 \to 0
\end{equation}
and the analogous one for the elliptic helix.

We now identify the degree two terms. To show 
$$
\Hom(\cO_{-1}, R_{\cO_1} \cO_0) = \Hom(\cE_{-1},R_{\cE_1} \cE_0)
$$ 
we apply to $\Hom(\cO_{-1},-)$ to the exact sequence in (\ref{eq:ROzero}) and $\Hom(\cE_{-1},-)$ to the corresponding exact sequence associated to the elliptic helix. This gives the exact sequences
\begin{multline}  \label{eq:deg2HomsSame}
    0 \to \Hom(\cO_{-1},\cO_0) \to \Hom(\cO_0,\cO_1)^* \otimes \Hom(\cO_{-1},\cO_1) \xrightarrow{m} \Hom(\cO_{-1},R_{\cO_1} \cO_0) \to 0 \\
    0 \to \Hom(\cE_{-1},\cE_0) \to \Hom(\cE_0,\cE_1)^* \otimes \Hom(\cE_{-1},\cE_1) \xrightarrow{m} \Hom(\cE_{-1},R_{\cE_1} \cE_0) \to 0
\end{multline}
The left two terms in these sequences are the same since $\mathbb{S}^{nc}(\underline{\cE})$ equals $\End(\underline{\cE})$ in degrees $\leq 2$ and furthermore, the two maps $m$ correspond to the multiplication of degree one morphisms by our calculation of degree one morphisms above.   A similar (or easier) argument identifies $\Hom(\cE_1,\cE_2) = \Hom(\cO_1,\cO_2)$ and $\Hom(R_{\cE_1}\cE_0,\cE_4) = \Hom(R_{\cO_1}\cO_0,\cO_4)$ as well as the multiplication maps to them. 
\end{proof}

\begin{definition}  \label{def:helixMutations}
We say that two elliptic helices $\underline{\cE},\underline{\cE}'$ are {\em mutations} of each other if they can be gotten from each other by a series of alterations of the form given in (1) of Proposition~\ref{prop:mutateHelix}. We similarly define mutations associated to Koszul AS-regular $\mathbb{Z}$-algebras of dimension three. 
\end{definition}

\begin{corollary}  \label{cor:MarkovIsMutateLines}
Any elliptic helix $\underline{\cE}$ which is of Markov type is a mutation of an elliptic helix consisting of line bundles. Thus, ${\sf Proj }\mathbb{S}^{nc}(\underline{\mathcal{E}})$ is equivalent to an elliptic quantum plane.
\end{corollary}
\begin{proof}
Note that there are three possible ways you can mutate every third bundle in $\cE$, and by Proposition~\ref{prop:mutateHelix}~(2), they correspond to the three possible mutations of the Markov triple of ranks. Since we can always find a sequence of these mutations of Markov triples to reduce to the triple $(1,1,1)$ by \cite[Lemma 5, p. 28]{cassels} we are done with a proof of the first statement. 

For the second statement, it suffices to show that if $\cE'$ is a mutation of the elliptic helix $\cE$, then $\mathbb{S}^{nc}(\cE), \mathbb{S}^{nc}(\cE')$ are Proj-equivalent. Let $\underline{\cO},\underline{\cO'}$ be the associated helices which, by Proposition~\ref{prop:mutationsCorrespond} are mutations of each other. By induction, we may assume that $\cO'$ is obtained from $\cO$ by mutating every third object so $\mathbb{S}^{nc}(\cE), \mathbb{S}^{nc}(\cE')$ have equal 3-Veroneses. Proj-equivalence now follows from \cite[Lemma~2.5]{ncquad}.
\end{proof}

\section{Constructing elliptic helices $\underline{\cE}$ with non-noetherian $\End(\underline{\cE})$}  \label{sec:seedGeneratesHelix}

In \cite[Theorem~8.1]{morph}, we constructed a single family of elliptic helices $\underline{\cE}$ whose endomorphism ring was non-noetherian. It happens to be equigenerated which made the analysis amenable to the simpler methods used there. In this section, we wish to construct more examples with $\End(\underline{\cE})$ non-noetherian including some which are not equigenerated. 

The key question is when a seed $(E_0,E_1,E_2)$ generates an elliptic helix in ${\sf Coh }X$.  Iterated mutation always produces objects in $D^{b}(X)$, but these objects need not be vector bundles concentrated in degree zero.  By \cite[Theorem 7.14]{morph}, in the present setting the seed generates an elliptic helix precisely when all recursively determined ranks $r_{i}$ and intermediate ranks $r_{i}'$ are positive.  The problem therefore reduces to a numerical one.  We begin with the numerical seed 

$$
([E_0],[E_1],[E_2]) = ({d_0 \choose r_0}, {d_1 \choose r_1}, {d_2 \choose r_2}).
$$
Recall Equation~(\ref{eq.matrixMutateFormula}) determines all the $d_i,r_i$ in terms of the numerical seed and the $\Delta_i$.  Our approach to a sufficient criterion to guarantee this is that Equation~(\ref{eq.degRkFormula}) ensures that $r_{i + 3n}$ is either a monotonic function of $n$ or is concave up. We will look for seeds which give a concave up function $r_n$ with a minimum at $n=1$. 

\begin{proposition}  \label{prop:convexityCriterion}
Suppose a numerical seed with $r_0, r_2 \geq r_1$ generates a sequence $r_i$ which is strictly increasing for $i \geq 1$, strictly decreasing for $i < 1$ and such that $r'_1,r'_2 >0$. Then the seed generates an elliptic helix with exponential growth.
\end{proposition}
\begin{proof}
Note first all $r_i \geq 0$. Now $r_i$ is increasing for $i \geq 1$ so 
$$r'_{i+2} = \Delta_{i+1}r_{i+1} - r_i >0$$
and a similar argument shows that $r'_j >0$ for all $j\leq 0$.  

To prove the resulting helix has exponential growth, we note that if this wasn't the case, our helix would have polynomial growth by Proposition \ref{prop.whyMarkov}.  Thus, by Theorem \ref{thm:onlyMarkov} and Theorem \ref{thm:BPBisKoszul}(1), the ranks $r_{i}$ are three-periodic and hence cannot be strictly increasing for $i \geq 1$.
\end{proof}

We wish to reduce checking piecewise monotonicity of $r_i$ to some finite condition as follows.

\begin{proposition} \label{prop:growthCriterionForGeneration}
Suppose the numerical seed has $r_0,r_2\geq r_1$ and $r'_1,r'_2 >0$. Suppose there exists a real number $l>1$ such that 
\begin{equation}  \label{eq:growthCondition}
    \Delta_i - \frac{\Delta_j}{l} > l, \quad \text{for all distinct } i,j \in \{0, 1, 2\}. 
\end{equation}
If $\frac{r_0}{r_1}, \frac{r_2}{r_1}>l$ then the seed generates an elliptic helix with exponential growth. 
\end{proposition}
\begin{proof}
We first show by induction that $\frac{r_{i+1}}{r_i} >l$ for $i\geq 1$. Note Equation~(\ref{eq.matrixMutateFormula}) gives for $i\geq 0$ 
$$ \frac{r_{i+3}}{r_{i+2}} = \frac{r_i}{r_{i+2}} - \Delta_{i+1} \frac{r_{i+1}}{r_{i+2}} + \Delta_i >\Delta_i - \frac{\Delta_{i+1}}{l} >l.$$
A similar argument shows $r_i$ is descending for $i < 1$ so Proposition~\ref{prop:convexityCriterion} applies.

\end{proof}

The approach suggests a natural starting point to look for seeds generating elliptic helices is to try $r_1 = 1$ in which case $L_1$ is a line bundle which we may as well normalise to be $\cO$.

\begin{example}  \label{eg:equigen}
Here we seek equigenerated examples, that is, where $\Delta_i = \Delta$ for all $i$. We thus start with the numerical seed $({ -\Delta \choose r_0}, {0 \choose 1}, {\Delta \choose r_2})$ so $\Delta_1 = \Delta_2 = \Delta$. From (\ref{eq.basicMutateFormulae}), we also want
\begin{equation*}
\Delta (r_0 + r_2) = \Delta_{02}  = \Delta^2 - \Delta  \Leftrightarrow \Delta = r_0 + r_2 + 1. 
\end{equation*}
We apply Proposition~\ref{prop:growthCriterionForGeneration} with $l = 2- \varepsilon$ where $\varepsilon >0$ is sufficiently small. Then (\ref{eq:growthCondition}) becomes 
\begin{equation*}
\Delta > \frac{2-\varepsilon}{1 - (2-\varepsilon)^{-1}} = \frac{(2-\varepsilon)^2}{1-\varepsilon} \to 4
\end{equation*}
as $\varepsilon \to 0$. Note that $r'_2 = \Delta - r_0 = r_2+1 >0, r'_1 = \Delta - r_2 = r_0 + 1 >0$. Thus the numerical seed  $({ -r_0-r_2-1 \choose r_0}, {0 \choose 1}, {r_0+r_2+1 \choose r_2})$ generates an elliptic helix with exponential growth whenever $r_0, r_2 \geq 2$ and $\gcd(r_0+r_2+1, r_i) = 1$ for $i = 0,2$.  


We next note that these examples are distinct from those constructed in \cite{morph}.  For, in the examples constructed in \cite[Theorem 1.3]{morph}, one has ranks $s_{-1}=d-2$, $s_{0}=1$, $s_{1}=1$ and $s_{2}=d-2$ where $d \geq 5$.  If this were among the examples constructed above, the minimum rank would be realized by $s_{0}$ or $s_{1}$.  However, in either case, there is a neighboring rank of $1$, so that the condition $r_{0}, r_{2} \geq 2$ cannot be achieved.    

We can generalise this to non-equigenerated examples as follows. Consider the numerical seed $({ -r_0-r_2-a \choose r_0}, {0 \choose 1}, {r_0+r_2+a \choose r_2})$ for some $a \in \mathbb{Z}_{\geq 1}$ so $\Delta_1 = \Delta_2 = r_0+r_2+a$ but now $\Delta_0 = a(r_0+r_2+a)$. Picking $l = a+1 - \varepsilon$ in Proposition~\ref{prop:growthCriterionForGeneration}, the above analysis shows we obtain an elliptic helix with exponential growth whenever $r_0,r_2 \geq a+1, \ r_0+r_2 > a^2+a+1$ and, of course  $\gcd(r_0+r_2+a, r_i) = 1$ for $i = 0,2$. 

\end{example}

\begin{example} \label{example.noneq}
Consider the numerical seed $({ -d-a-1 \choose d+a-r}, {0 \choose 1}, {d \choose r})$ where $a,d,r \in \mathbb{N}$ are chosen so that $\gcd(d+a+1,d+a-r) = 1 = \gcd(d,r)$. Clearly we must have $d+a-r >0$. Note $\Delta_0 = d-(a+1)r, \Delta_1 = d+a+1, \Delta_2 = d$. To apply Proposition~\ref{prop:growthCriterionForGeneration} we first need 
$$ r'_1 = \Delta_2 r_1 - r_2 = d-r >0, \quad r'_2 = \Delta_1 r_1 - r_0 = r+1 >0.$$
We also need to impose the condition (\ref{eq:growthCondition}) for some $l$. Picking $l =2$, the set of inequalities amount to the strictest one
$$\min_i \Delta_i - \max_j \frac{\Delta_j}{2} = d-(a+1)r - \frac{d+a+1}{2}>2$$
which is equivalent to the condition $d>(a+1)(2r+1)+4$.  Finally, we impose the condition $\min\{ d+a-r, r \} > 2$. When these hold, we obtain an elliptic helix with exponential growth.  

\end{example}
These examples suggest that elliptic helices are relatively plentiful, but finding them is not so easy, let alone giving a complete classification.

\end{document}